\documentclass[a4paper,fleqn]{cas-sc}
\usepackage[utf8]{inputenc}
\usepackage{optprog}
\usepackage{todonotes}
\usepackage[numbers]{natbib}
\usepackage{graphicx}
\usepackage{latexsym}
\usepackage{booktabs}
\usepackage{bbm}
\usepackage{float}
\usepackage{subcaption}
\usepackage{booktabs,array}
\usepackage{enumitem}
\usepackage{verbatim}
\usepackage{amsmath}
\usepackage{amsthm}
\usepackage{thmtools}
\usepackage{mathtools}
\usepackage{amssymb}
\usepackage{cleveref}
\usepackage{xcolor}
\newcommand{\R}{\mathbb{R}}

\newcommand{\X}{\mathbf{X}}

\renewcommand{\L}{\mathcal{L}}
\newcommand{\I}{\mathbf{I}}
\newcommand{\1}{\mathbf{1}}
\newcommand{\A}{\mathbf{A}}
\newcommand{\x}{\mathbf{x}}
\DeclareMathOperator{\mc}{MAXCUT}

\newcommand{\bz}{\mathbf{0}}
\newcommand{\B}{\mathbf{B}}
\newcommand{\bS}{\mathbf{S}}
\newcommand{\Y}{\mathbf{Y}}
\newcommand{\bu}{\mathbf{u}}
\newcommand{\bw}{\mathbf{w}}
\newcommand{\bZ}{\mathbf{Z}}
\newcommand{\bv}{\mathbf{v}}
\newcommand{\bm}{\mathbf{m}}
\newcommand{\ind}{\mathbbm{1}}

\newcommand*{\myQED}{\hfill \scalebox{0.59}{\fbox{\rule{0.08mm}{0pt}\rule{0pt}{1.3mm}}}}%

\newtheorem{example}{Example}
\newenvironment{proofsketch}{%
  \par\noindent\textit{Proof sketch.}\ }%
  {\hfill\par}

\newtheorem{theorem}{Theorem}
\newtheorem{lemma}[theorem]{Lemma}
\newtheorem{proposition}[theorem]{Proposition}
\newtheorem{corollary}[theorem]{Corollary}
\newtheorem{definition}[theorem]{Definition}
\usepackage{tikz}
\usetikzlibrary{graphs,graphs.standard,quotes,arrows.meta}
\usetikzlibrary{arrows}
\usetikzlibrary{shapes.geometric}
\DeclareMathOperator{\rank}{rank}
\DeclareMathOperator{\Diag}{Diag}
\DeclareMathOperator{\diag}{diag}
\DeclareMathOperator{\range}{range}
\DeclareMathOperator{\tr}{tr}
\tikzset{cedge/.style={draw=#1}
}
\newcolumntype{P}[1]{>{\centering\arraybackslash}p{#1}}
\newcommand{\ignore}[1]{}
\newdefinition{remark}{Remark}
\newcounter{propCounter}

\newlist{Properties}{enumerate}{2}
\setlist[Properties]{label=Property \arabic*.,itemindent=*}

\begin{document}
\shorttitle{On exactness of SDP relaxation for the Maximum Cut Problem}    

\shortauthors{A. Bhardwaj, H. Gogoi, V. Narayanan, and A. Pathapati} 

\title[mode = title]{On exactness of SDP relaxation for the Maximum Cut Problem}

\author[1]{Avinash Bhardwaj}[orcid=0000-0002-9690-1705]
\cormark[1]
\ead{abhardwaj@iitb.ac.in}
\author[1]{Hritiz Gogoi}
\ead{hgogoi@iitb.ac.in}
\author[1]{Vishnu Narayanan}
\ead{vishnu@iitb.ac.in}
\author[1]{Abhishek Pathapati}[orcid=0009-0009-2932-167X]
\ead{apathapati@iitb.ac.in}

\affiliation[1]{organization = {Department of Industrial Engineering and Operations Research},
addressline={IIT Bombay},
city={Mumbai, Maharashtra},
postcode={400076},
country={India}
}
\cortext[cor1]{Corresponding author}

\maketitle 

\begin{abstract}
Semidefinite programming (SDP) provides a powerful relaxation for the maximum cut problem. In this work, we characterize a few classes of graphs for which the SDP relaxation is exact. For each of these graph classes, we establish conditions for uniqueness of the SDP optimum. We complement these findings by identifying two graph operations that preserve the solution rank, and in turn exactness. These results reveal how the SDP relaxation for the maximum cut problem can remain exact in arbitrarily large graphs, owing to the presence of a small structural core that governs exactness. We further address two open problems posed by Mirka and Williamson (2024), by demonstrating that uniqueness of the maximum cut partition in exact relaxation does not imply uniqueness of the SDP optimum, and that exact relaxation with multiple optimal partitions may admit optimal SDP solutions lying outside the convex hull of rank-1 reference solutions.
\end{abstract}

\section{Introduction}\label{intro}
The Maximum-Cut (Max-Cut) problem asks the following: given a graph $G = (V, E)$ with non-negative weights assigned to its edges, find an optimal partition of the vertex set $V$ into two disjoint subsets such that the total weight of the edges crossing the partition (i.e., edges with endpoints in different subsets) is maximized. Over the years, this problem has generated immense interest. It is one of the earliest problems used to demonstrate the $NP$-hardness of computational problems \citep{karp2009reducibility}, and the first optimization problem for which a semidefinite programming relaxation was used to obtain a polynomial-time algorithm with a provable constant-factor approximation guarantee \citep{goemans1995improved}.
Apart from its theoretical significance, the Max-Cut problem finds applications across various disciplines and industries including statistical physics \citep{barahona1988application}, VLSI design \citep{barahona1988application, franuke2011via} and machine learning \citep{wang2013semi}. 

A semidefinite relaxation for the Max-Cut problem was introduced in~\citet{goemans1995improved, poljak1995nonpolyhedral}, leading to a polynomial-time approximation algorithm. Given an input graph, the relaxation produces a positive semidefinite matrix as the optimal solution. A feasible cut can then be obtained by applying a hyperplane rounding technique to this matrix, achieving an expected value of at least $\alpha_{\text{GW}} =\frac{2}{\pi} \displaystyle\min_{0 \leq \theta \leq \pi} \frac{\theta}{1 - \cos \theta} \approx  0.87856$ times the value of the maximum cut \citep{goemans1995improved}. This approximation ratio is optimal under the assumption of the Unique Games Conjecture and $P \neq NP$ \citep{khot2007optimal}. The Max-Cut SDP relaxation is exact if and only if a rank-1 optimal primal solution exists. In practice, however, exactness is difficult to certify: the set of optimal solutions may contain higher-rank (rank $>$ 1) solutions and solvers often return them. It is also possible that sub-optimal feasible solutions are recovered from hyperplane rounding even in cases where the relaxation is exact. Notably, Karloff \citep{karloff1996good} identified a class of graphs for which the semidefinite programming relaxation yields an exact solution; yet, the approximation ratio attained in these cases corresponds to the worst-case. From a theoretical standpoint, the study of conditions for exactness provides insight into the geometry of the feasible region (elliptope), revealing how combinatorial optima interact with convex relaxations. Understanding when higher-rank optima arise, and how uniqueness of solutions is preserved links Max-Cut to broader themes concerning the rank of optimal solutions in semidefinite relaxations. This motivates a systematic investigation of graph classes and structural principles that govern exactness and uniqueness in the Max-Cut SDP relaxation.

\subsection{Overview and our contributions}\label{overview}

Semidefinite programming has played a foundational role in combinatorial optimization since the seminal work of Lovász on the theta function, which provided an SDP-based characterization and bounds for the Shannon capacity of a graph and related problems \citep{lovasz1979shannon}. 
Soon after, similar semidefinite ideas were brought to bear on the Max-Cut problem, with the study of non-polyhedral relaxations for the Max-Cut problem initiated by Svatopluk Poljak and collaborators. Mohar and Poljak~\citep{moharpoljak1990eigenvalues} first proposed an eigenvalue-based upper bound, which was subsequently refined (referred to as the Delorme--Poljak bound) and evaluated on various graph families by Delorme and Poljak~\citep{delorme1993alaplacian,delorme1993combinatorial,delorme1993performance}. These evaluations included the identification of several classes of graphs for which the Delorme-Poljak bound is tight, which we summarize in \Cref{knowngraphs}. They further proved that the Exact Weighted Graph Problem— that is, deciding whether the Delorme-Poljak bound coincides with the actual value of Max-Cut given a graph with rational weights—is $NP$-Hard \citep{delorme1993combinatorial}.
Building on this line of work, Poljak and Rendl~\citep{poljak1995nonpolyhedral} introduced a semidefinite program to compute the Delorme--Poljak bound. Laurent and Poljak \citep{laurent1995positive} investigated the feasible region of the semidefinite relaxation of the Max-Cut problem, the elliptope, and provided a complete characterization of objective functions of the form \( C = bb^\top \), for some vector \( b \), for which the relaxation is exact. Around the same time, Goemans and Williamson~\citep{goemans1995improved} introduced an equivalent semidefinite relaxation (called the GW-relaxation) and established that its approximation guarantee is at least~$\alpha_{\text{GW}}$ using hyperplane rounding. Karloff \citep{karloff1996good} showed that the approximation ratio of GW-relaxation is exactly $\alpha_{GW}$. Tightness of the Goemans--Williamson approximation ratio and its dependence on spectral properties have been further investigated using eigenvalue-based techniques; see, for example, \citep{alon2000bipartite} \citep{alon2001constructing}. Sojoudi and Lavaei \citep{sojoudi2014exactness} investigated the exactness of semidefinite relaxation for general quadratically constrained quadratic programs (QCQPs), of which the Max-Cut problem is a special case. De Silva and Tuncel \citep{de2019strict} studied strict complementarity phenomenon over elliptopes, when SDP relaxation is exact. More recently, Hong et al. \citep{hong} characterized conditions under which GW-relaxation is exact on weighted cycle graphs. They also derived bounds on the rank of the optimal solution for the vertex-sum of two graphs. Mirka and Williamson \citep{mirka2024max} investigated the uniqueness of the optimal rank-1 solution for the GW-relaxation. We summarize the graph classes identified in the literature for which the Delorme-Poljak eigenvalue bound is tight, or equivalently, the GW-relaxation is exact, in \Cref{knowngraphs}.

In this manuscript, we refer to the uniqueness of the optimal solution in the context of the primal problem (see \Cref{relax}); the dual optimal solution is always unique \citep{delorme1993alaplacian, de2019strict}. We call a graph~$G$ (or a family of graphs) \emph{exact} if the GW-relaxation is exact, or equivalently, if the Delorme--Poljak bound is tight. Although some exact graph families are known (summarized in \Cref{knowngraphs}), the exactness of the Max-Cut SDP relaxation is not fully understood. This manuscript addresses the following aspects of SDP exactness for the Max-Cut problem:

\begin{table}[h]
\centering
\caption{Known graph families for which the Max-Cut SDP relaxation is exact.\vspace*{0.5mm}}
\label{knowngraphs}
\begin{tabular}{p{4cm} p{9cm}}
\hline
\textbf{Source} & \textbf{Exact graph families} \\
\hline
Delorme--Poljak~\citep{delorme1993combinatorial,delorme1993alaplacian,delorme1993performance} 
& Cartesian product of two exact graphs; 
Tensor product $K_m \times K_n$ (exact if $n \leq m$, $m$ even); 
Line graphs of bipartite $(r,s)$-semiregular graphs ($r,s$ even); 
Line graphs of $K_{4r+1}$; 
Complement of line graphs of bipartite $G(m,n)$ (exact if $m \leq n$, $m$ even); 
Graph $G$ exact if all its bi-connected components are exact; 
Amalgam $G''$ of exact $G,G'$ exact if Max-Cut partitions coincide on $V \cap V'$; 
Cayley graphs for $\mathcal{G}=\mathbb{Z}_2^m$ with weights $w_{uv}=w_{u+v}$. \\[4pt]
\midrule
Laurent--Poljak~\citep{laurent1995positive} 
& Weighted complete graphs $K_n$ with weights $w_{ij}=m_i m_j$; exact if and only if either there exists $S \subseteq [n]$ such that $\sum_{j \in S} m_j =  \sum_{i \notin S} m_i$ or there exists $j$ such that $m_j > \sum_{i \neq j} m_i$. \\[4pt]
\midrule
Karloff~\citep{karloff1996good} 
&  Let, $J(m,t,b)$ be a graph such that: 
    \begin{enumerate}
        \item  Vertex set is the set of all $\binom{m}{t}$ t-element subsets of $\{1,2, \ldots, m\}$.
        \item  Two subsets $S$ and $T$ are adjacent if and only if $\left | S \cap T \right | = b$ and $S \neq T$.
    \end{enumerate}
    Then, $J(m,t,b)$ is exact if $t = \dfrac{m}{2}$, $b \leq \dfrac{m}{12}$ and $m$ is even.  \\[4mm]
\midrule
Hong et al.~\citep{hong}
& Weighted cycle graphs $C_n$ are exact if and only if at least one of the following is true:
        \begin{enumerate}
            \item There are an even number of positively weighted edges in $C_n$.
            \item There exists an $m$ such that $\frac{1}{\left|w_{m,m+1}\right|} \geq \sum\limits_{i \neq m} \frac{1}{\left|w_{i,i+1}\right|}$.
        \end{enumerate}
        They also showed uniqueness of the SDP solution. \\[4pt]
\midrule
Mirka--Williamson~\citep{mirka2024max} 
& Non-negatively weighted bipartite graphs (uniqueness of rank-1 solution).\\ 
\hline
\end{tabular}
\end{table}
\begin{enumerate}
    
    \item \textit{New exact families with uniqueness conditions.} We identify a few families of graphs for which the GW-relaxation is exact. For each of these graph families, we provide conditions that guarantee uniqueness of the optimal primal solution. We derive a rank equality for optimal primal--dual pairs as a key tool in establishing uniqueness. This expands the study of uniqueness of the rank-1 SDP optimum beyond the previously known cases of weighted cycles \citep{hong} and bipartite graphs \citep{mirka2024max}. For each class, we determine the optimal objective value by explicitly constructing a maximum cut. In \Cref{tab:exact-examples}, we present a few representative classes of graphs for which the exactness of the GW-relaxation, and uniqueness of the optimal solution can be explained by our results; many additional families can be derived similarly.

\begin{table}[h]
\centering
\caption{\centering Examples of graph classes for which our results explain exactness of Max-Cut SDP relaxation and uniqueness of the solution.\\[2mm]}
\label{tab:exact-examples}
\begin{tabular}{p{6.75cm}P{5.3cm}}
\toprule 
\textbf{Graph family} & \textbf{Unique? (Reference)}\\
\bottomrule\\\addlinespace[-1em]
$C_{m}\vee C^c_{m}$ & Yes (\Cref{claim:1}) \\[2pt]
$K_{2m }\vee C_{2n}$ with $n \ge m \ge 2$ & Yes (\Cref{claim:2}) \\[2pt]
$C_m \vee \overline{K_n}$ (m-gonal n-cone), with $n \geq 4$ & Yes (\Cref{claim:2}) \\[2pt]
$P_m \vee \overline{K_n}$ (fan graph), with $n \geq 4$ & Yes (\Cref{claim:2}) \\ [2pt]
$(P_{n} \vee C_n) \vee K(m-1,m+1,2m)$ & No  (\Cref{claim:3})\\[2pt]
$C_4\vee K(m_1,m_2,m_3)$, with $m_1+m_2 = m_3$ & No (\Cref{claim:3})\\ [2pt]
$K(1,2,3,4) \bullet G$ where $G$ is any unweighted graph & No (\Cref{prop1,cor:complete}, \Cref{thm:1})\\ 
\hline
\end{tabular}
\end{table}

\item \textit{Exactness under graph operations.} We prove that the \textit{graph join} of two graphs with the same number of vertices is exact. We show that if the first factor in the \textit{lexicographic product} of a finite collection of graphs is exact, then the rank of its optimal solution(s) is preserved, ensuring that the resulting product graph remains exact. Furthermore, we prove that a graph is exact if and only if its \textit{vertex-split graph} is exact, with the vertex-splitting operation also preserving the rank of optimal solutions. Using these results, we illustrate how exactness of Max-Cut SDP relaxation on a large graph may be dictated by the presence of a small structural core.
    
\item \textit{Higher-rank phenomena in exact graphs.} We address open questions of Mirka--Williamson~\citep{mirka2024max} by showing that, for certain classes of graphs, higher-rank optimal solutions necessarily occur even when the SDP relaxation is exact. Within these graph families, we demonstrate that uniqueness of the maximum cut partition in exact relaxation does not imply uniqueness of the SDP optimum, and that exact relaxation with multiple maximum cut partitions may admit optimal points lying outside the convex hull of rank-1 reference solutions. These observations may lay the groundwork for describing the convex hull of optimal solutions in exact GW-relaxation. Finally, we fully characterize the conditions for exactness and uniqueness of the rank-1 SDP optimal solution for complete $k$-partite graphs.
\end{enumerate}

\begin{figure}
    \centering
\scalebox{0.7}{%
\begin{tikzpicture}
  \begin{scope}[every node/.style={circle,fill,draw,inner sep=0, minimum size=0.3em}]
    \node (A1) at (3.0000, 0.0000) {};
    \node (A2) at (2.656371, 1.394168) {};
    \node (A3) at (1.701805, 2.470597) {};
    \node (B1) at (0.330667, 2.981721) {};
    \node (B2) at (-1.114987, 2.785099) {};
    \node (B3) at (-2.345492, 1.870471) {};
    \node (C1) at (-2.924785, 0.667563) {};
    \node (C2) at (-2.924785, -0.667563) {};
    \node (C3) at (-2.345492, -1.870471) {};
    \node (D1) at (-1.114987, -2.785099) {};
    \node (D2) at (0.330667, -2.981721) {};
    \node (D3) at (1.701805, -2.470597) {};
    \node (E) at (2.656371, -1.394168) {};
  \end{scope}
  \begin{scope}[>={Stealth[black]},
                  every node/.style={},
                  every edge/.style={draw=black,thick}]
        \path [-] (A1) edge[cedge=cyan] node {} (B1);
        \path [-] (A1) edge[cedge=cyan] node {} (B2);
        \path [-] (A1) edge[cedge=cyan] node {} (B3);
        \path [-] (A1) edge[cedge=cyan] node {} (C1);
        \path [-] (A1) edge[cedge=cyan] node {} (C2);
        \path [-] (A1) edge[cedge=cyan] node {} (C3);
        \path [-] (A1) edge[cedge=cyan] node {} (D1);
        \path [-] (A1) edge[cedge=cyan] node {} (D2);
        \path [-] (A1) edge[cedge=cyan] node {} (D3);
        \path [-] (A1) edge[cedge=green!40!gray] node {} (E);
        \path [-] (A2) edge[cedge=red!40!gray] node {} (A3);
        \path [-] (A2) edge[cedge=cyan] node {} (B1);
        \path [-] (A2) edge[cedge=cyan] node {} (B2);
        \path [-] (A2) edge[cedge=cyan] node {} (B3);
        \path [-] (A2) edge[cedge=cyan] node {} (C1);
        \path [-] (A2) edge[cedge=cyan] node {} (C2);
        \path [-] (A2) edge[cedge=cyan] node {} (C3);
        \path [-] (A2) edge[cedge=cyan] node {} (D1);
        \path [-] (A2) edge[cedge=cyan] node {} (D2);
        \path [-] (A2) edge[cedge=cyan] node {} (D3);
        \path [-] (A2) edge[cedge=green!40!gray] node {} (E);
        \path [-] (A3) edge[cedge=cyan] node {} (B1);
        \path [-] (A3) edge[cedge=cyan] node {} (B2);
        \path [-] (A3) edge[cedge=cyan] node {} (B3);
        \path [-] (A3) edge[cedge=cyan] node {} (C1);
        \path [-] (A3) edge[cedge=cyan] node {} (C2);
        \path [-] (A3) edge[cedge=cyan] node {} (C3);
        \path [-] (A3) edge[cedge=cyan] node {} (D1);
        \path [-] (A3) edge[cedge=cyan] node {} (D2);
        \path [-] (A3) edge[cedge=cyan] node {} (D3);
        \path [-] (A3) edge[cedge=green!40!gray] node {} (E);
        \path [-] (B1) edge[cedge=red!40!gray] node {} (B2);
        \path [-] (B1) edge[cedge=cyan] node {} (C1);
        \path [-] (B1) edge[cedge=cyan] node {} (C2);
        \path [-] (B1) edge[cedge=cyan] node {} (C3);
        \path [-] (B1) edge[cedge=cyan] node {} (D1);
        \path [-] (B1) edge[cedge=cyan] node {} (D2);
        \path [-] (B1) edge[cedge=cyan] node {} (D3);
        \path [-] (B1) edge[cedge=green!40!gray] node {} (E);
        \path [-] (B2) edge[cedge=red!40!gray] node {} (B3);
        \path [-] (B2) edge[cedge=cyan] node {} (C1);
        \path [-] (B2) edge[cedge=cyan] node {} (C2);
        \path [-] (B2) edge[cedge=cyan] node {} (C3);
        \path [-] (B2) edge[cedge=cyan] node {} (D1);
        \path [-] (B2) edge[cedge=cyan] node {} (D2);
        \path [-] (B2) edge[cedge=cyan] node {} (D3);
        \path [-] (B2) edge[cedge=green!40!gray] node {} (E);
        \path [-] (B3) edge[cedge=cyan] node {} (C1);
        \path [-] (B3) edge[cedge=cyan] node {} (C2);
        \path [-] (B3) edge[cedge=cyan] node {} (C3);
        \path [-] (B3) edge[cedge=cyan] node {} (D1);
        \path [-] (B3) edge[cedge=cyan] node {} (D2);
        \path [-] (B3) edge[cedge=cyan] node {} (D3);
        \path [-] (B3) edge[cedge=green!40!gray] node {} (E);
        \path [-] (C1) edge[cedge=red!40!gray] node {} (C2);
        \path [-] (C1) edge[cedge=red!40!gray] node {} (C3);
        \path [-] (C2) edge[cedge=red!40!gray] node {} (C3);
        \path [-] (D1) edge[cedge=red!40!gray] node {} (D3);
        \path [-] (D1) edge[cedge=cyan] node {} (D2);
        \path [-] (D2) edge[cedge=red!40!gray] node {} (D3);
        \path [-] (C1) edge[cedge=green!40!gray] node {} (E);
    \end{scope}
    \end{tikzpicture}} 
\caption{\centering \small Graph $\mathcal{G}$ with edge colors indicating weights: \textcolor{red!40!gray}{Red} (1), \textcolor{cyan}{Blue} (2) \& \textcolor{green!40!gray}{Green} (6).}
    \label{fig:GRAPHEXAMPLE}
\end{figure}

\begin{example}\label{example1}
Consider the graph $\mathcal{G}$ in \Cref{fig:GRAPHEXAMPLE}. Can we examine the structure of $\mathcal{G}$ and decompose it into simpler components to reveal whether $\mathcal{G}$ is exact? It is thus, we are motivated to explore the following: Given a graph, can its exactness be inferred from its structure? In this manuscript, we contribute to the set of tools and techniques to address this question.
\end{example} 

\section{Preliminaries}\label{sec3}

$\mathbb{R}^n$, $\mathbb{Z}^n$, $\mathbb{R}_{>0}^n$, and $\mathbb{Z}_{>0}^n$ denote the sets of $n$-dimensional vectors with real, integer, positive real, and positive integer entries, respectively. $[n] := \{1, 2, \ldots, n\}$. For a matrix $\mathbf{A} \in \mathbb{R}^{n \times n}$, the notation $\mathbf{A} \succeq \mathbf{0}$ indicates that $\mathbf{A}$ is positive semidefinite. $\mathcal{N}(\mathbf{A})$, $\operatorname{tr}(\mathbf{A})$, and $\lambda_{\max}(\mathbf{A})$ denote the nullspace, trace, and largest eigenvalue of $\mathbf{A}$, respectively. Given a subset $S \subseteq V$, we denote by $\mathbbm{1}_S \in \{0,1\}^{|V|}$ the indicator vector of $S$, i.e. $i$th entry of $\mathbbm{1}_S$ equals $1$ if $i \in S$ and $0$ otherwise. $\langle \cdot , \cdot \rangle$ denotes the Frobenius inner product in the space of symmetric matrices, defined by $\langle \mathbf{A}, \mathbf{B} \rangle = \operatorname{tr}(\mathbf{A} \mathbf{B})$. Given a graph $G$, $\mathrm{MAXCUT}(G)$ denotes the value of the maximum cut, and $G_1 \cong G_2$ indicates that graphs $G_1$ and $G_2$ are isomorphic. The complete $k$-partite graph is denoted by $K(m_1, m_2, \ldots, m_k)$, where $m_1 \le m_2\leq \ldots \leq m_k$ are the cardinalities of its independent sets. $K_n$, $\overline{K}_n$, $C_n$, and $P_n$ denote the complete graph, the edgeless graph, the cycle graph, and the path graph on $n$ vertices, respectively. Let $G^c$ denote the complement graph of $G$. We say that the complete graph $K_n$ is \emph{weighted by} $\bm \in \mathbb{R}_{>0}^n$ if its edge weights satisfy $w_{ij} = m_i m_j$ for all $i, j \in [n]$. A graph $G = (V, E)$ is said to be \emph{unweighted} if $w_{ij} = 1$ for all $(i,j) \in E$. Unless stated otherwise, all graphs considered in this manuscript are simple and unweighted.\\

The Laplacian of a Graph $G$ (denoted by $\L_{G}$) is the matrix with elements $x_{ij}$ defined as: $x_{ij} = -w_{ij}$ if $i\neq j$ and $x_{ii} = \sum\limits_{i \neq j} w_{ij}$, where $i, j \in [n]$. We write $\deg_G(v)$ for the degree of $v$ in $G$. For $U \subseteq V$, we write $G[U]$ for the
subgraph of $G$ induced by $U$. The \textit{graph join} of two disjoint graphs \( G_1 = (V_1, E_1) \) and \( G_2 = (V_2, E_2) \), denoted as \( G_1 \vee G_2 \), is the graph: $G_1 \vee G_2 := (V_1 \cup V_2, E_1 \cup E_2 \cup \{ (u, v) \mid u \in V_1, v \in V_2 \} )$. Given a weighted graph $G (V,E, \mathbf{w})$ on $n$ vertices and integers $p_i \geq 1, i \in [n]$, the \textit{vertex-split graph} $\widetilde{G}$ is obtained by replacing each vertex $i$ with $p_i$ distinct copies. For each edge $ij \in E$ all possible edges are added between the copies of $i$ and the copies of $j$, and every edge has weight $w_{ij}/(p_i p_j)$, where $w_{ij}$ is the weight of edge $ij \in E$. The \textit{lexicographic product} of graphs $G_1=(V_1,E_1)$ and $G_2=(V_2,E_2)$ is the graph $G_1 \bullet G_2$ with vertex set $V_1 \times V_2$, where two vertices $(u,v)$ and $(u',v')$ are adjacent if either $uu' \in E_1$, or $u=u'$ and $vv' \in E_2$. When \(G_1\) and \(G_2\) are weighted graphs with the same uniform edge weight \(w\), we assign the same weight \(w\) to every edge of the lexicographic product \(G_1 \bullet G_2\). For a vector \(\bv \in \mathbb{R}^n\), we define \(\Diag(\bv) \in \mathbb{R}^{n \times n}\) as the diagonal matrix with \((i,i)\)-entry equal to \(v_i\) and all off-diagonal entries zero.

\subsection{Relaxations for Max-Cut}\label{relax}
The Max-Cut problem can be formulated as a quadratic integer program, which upon vector lifting yields the GW-relaxation. In its standard primal form, the relaxation is given by:

\begin{optprog*}\label{mcut}
    maximize $z_{\text{MAXCUT}}$ = &\objective{\frac{1}{4} \langle \L_G,\mathbf{X} \rangle}\\
    (MAXCUT SDP) \qquad subject to & X_{ii} = 1,~~ \forall i \in V\\
    & \mathbf{X} \succeq \textbf{0} \hspace*{1.4cm}
\end{optprog*}
Imposing $\operatorname{rank}(\X)=1$ recovers the exact quadratic integer program.\\[2mm]$\mathrm{MAXCUT\ SDP}$ has the following dual:

\begin{optprog*}
    minimize $z_{\mathrm{MAXCUTDUAL}}$ =
        & \objective{\sum_{i \in V} y_i} \\
    (MAXCUT DUAL SDP) \qquad subject to
        & \begin{aligned}[t]
            \mathbf{S} &= \operatorname{Diag}(y) - \frac{1}{4}\L_G, \\
            \mathbf{S} &\succeq \mathbf{0}, \\
            y &\in \mathbb{R}^{|V|}.
          \end{aligned}
          \label{prob:maxcutdual}
\end{optprog*}

The $\mathrm{MAXCUT\ SDP}$ satisfies strong duality \citep{vanboyd}, ensuring equality between its primal and dual optima, as well as complementary slackness. The GW-relaxation can be expressed in an equivalent eigenvalue form~\citep{poljak1995nonpolyhedral}, often referred to as the Delorme--Poljak formulation. Let $\phi(G)$ denote its optimal value. $\phi(G)$ exhibits the following  properties~\citep{delorme1993alaplacian,delorme1993combinatorial, delorme1993performance}, which, by equivalence, are inherited by the $\mathrm{MAXCUT\ SDP}$ optimum $z^*_{\mc}$. 

\begin{itemize}
    \item \textbf{Upper bound:} $\phi(G) \geq \mathrm{MAXCUT(G)}$.
    \item \textbf{Scaling:} If each weight $w_{ij}$ of $G$ is multiplied by some positive real $k$, then $\phi(G)$ is also multiplied by $k$.
    \item \textbf{Monotonicity:} If $G$ and $G'$ have the same vertices and the weight functions $w$ and $w'$ satisfy $w_{ij} \leq w'_{ij}$ for all pairs of vertices, then $\phi(G) \leq \phi(G')$.
    \item \textbf{Vertex-Split invariance:} $\phi(\widetilde{G}) = \phi(G)$ for every vertex-split graph $\widetilde{G}$ of $G$. 
\end{itemize}

\section{Main Results}\label{sec4}

In this section, we present new families of unweighted graphs for which the $\mathrm{MAXCUT\ SDP}$ is exact. We also provide condition for uniqueness of the optimal rank-1 solution for each class. We adopt the following setup for this section. Let $G_K$ be the complete bipartite graph $K(m_1,m_2)$ with bipartition $(V_A,V_B)$, where $|V_A| = m_1$ and $|V_B| = m_2$. Let $G_A = (V_A,E_A)$ and $G_B = (V_B,E_B)$ be graphs defined on the vertex sets $V_A$ and $V_B$, respectively. We define $G' \;=\; G_K \cup G_A \cup G_B$.\vspace{1mm}

We begin with the balanced case, i.e.\ \(m_1=m_2=n\). While the complete bipartite graph \( K(n,n) \) is exact and admits a unique optimal solution, the complete graph \( K_{2n} \) is also exact (see \Cref{knowngraphs}) but admits multiple rank-1 optimal solutions corresponding to distinct \( n \)-by-\( n \) bipartitions. Thus, as edges are added within the two independent sets \(V_A\) and \(V_B\) of $K(n,n)$, exactness of the relaxation persists but the uniqueness property can be lost. \Cref{prop1} addresses this and identifies a sufficient condition when uniqueness is retained. \ignore{We give a sketch of the proof of \Cref{claim:1} and defer the full proof to \Cref{proof1}.}

\begin{proposition}\label{prop1}
    \label{claim:1}
    $G'$ is exact if $|V_A| = |V_B|$. Furthermore, the optimal solution of $\mathrm{MAXCUT\ SDP}$ for such $G'$  is unique if and only if at least one of $G_A^c$ and $G_B^c$ is connected.
\end{proposition}

\begin{corollary}
The graph join of any two graphs with the same number of vertices is exact. In particular, $G \vee G$ is exact. Moreover, $G \vee G$ admits a unique optimal $\mathrm{MAXCUT\ SDP}$ solution if $G^c$ is connected.
\end{corollary}

\begin{corollary}\label{claim:3}
    $G'$ is exact when both $G_A$ and $G_B$ contain balanced complete bipartite spanning subgraphs; that is, $G_A$ contains $K(|V_A|/2,\; |V_A|/2)$ and $G_B$ contains $K(|V_B|/2,\;|V_B|/2)$. Moreover, a higher-rank optimal $\mathrm{MAXCUT\ SDP}$ solution always exists for such $G'$. 
\end{corollary}
\begin{figure}
\begin{center}
\begin{tikzpicture}[scale=0.70, transform shape]

\begin{scope}[xshift=-4.5cm, yshift=0cm]
    \begin{scope}[every node/.style={circle,thick,draw}]
        \node (A) at (0,2) {A};
        \node (B) at (2.7,2.5) {B};
        \node (C) at (5.3,2) {C};
        \node (F) at (0,-0.9) {F};
        \node (E) at (2.7,-1.4) {E};
        \node (D) at (5.3,-0.9) {D};
    \end{scope}

    \begin{scope}[every edge/.style={draw=black,very thick}]
        \foreach \u in {A,B,C}
            \foreach \v in {D,E,F}
                \draw (\u) -- (\v);
        \draw (A)--(B);
        \draw (D)--(E);
        \draw (B)--(C);
        \draw (F)--(E);
    \end{scope}

    \node at (2.7,-2.2) {(a)};
\end{scope}

\begin{scope}[xshift=4.5cm, yshift=0cm]
    \begin{scope}[every node/.style={circle,thick,draw}]
        \node (A) at (0,2) {A};
        \node (B) at (2.7,2.5) {B};
        \node (C) at (5.3,2) {C};
        \node (F) at (0,-0.9) {F};
        \node (E) at (2.7,-1.4) {E};
        \node (D) at (5.3,-0.9) {D};
    \end{scope}

    \begin{scope}[every edge/.style={draw=black,very thick}]
        \foreach \u in {A,B,C}
            \foreach \v in {D,E,F}
                \draw (\u) -- (\v);
        \draw (A)--(B);
        \draw (B)--(C);
        \draw (A)--(C);
        \draw (D)--(E);
    \end{scope}

    \node at (2.7,-2.2) {(b)};
\end{scope}

\begin{scope}[xshift=0cm, yshift=-4.8cm]
    \begin{scope}[every node/.style={circle,thick,draw}]
        \node (A) at (0,2) {A};
        \node (B) at (2.7,2.5) {B};
        \node (C) at (5.3,2) {C};
        \node (F) at (0,-0.9) {F};
        \node (E) at (2.7,-1.4) {E};
        \node (D) at (5.3,-0.9) {D};
    \end{scope}

    \begin{scope}[every edge/.style={draw=black,very thick}]
        \foreach \u in {A,B,C}
            \foreach \v in {E,F}
                \draw (\u) -- (\v);
        \draw (A)--(B);
        \draw (A)--(C);
        \draw (B)--(C);
        \draw (D)--(E);
        \draw (A)--(D);
        \draw (B)--(D);
    \end{scope}

    \node at (2.7,-2.2) {(c)};
\end{scope}

\end{tikzpicture}
\captionof{figure}{
Three graphs illustrating distinct MAXCUT SDP behavior:
(a) exact with non-unique optimum,
(b) exact with unique optimum,
(c) non-exact.
}\label{fig:structure_vs_size}
\end{center}
\end{figure}

The sensitivity of exactness and uniqueness of the $\mathrm{MAXCUT}$ SDP to graph structure
can be seen from \Cref{prop1}.
This phenomenon is illustrated in \Cref{fig:structure_vs_size}, which presents three graphs
on the same vertex set, differing only by minimal edge modifications.
The graph \Cref{fig:structure_vs_size}(b) is obtained from
\Cref{fig:structure_vs_size}(a) by replacing the edge $(E,F)$ with $(A,C)$.
While \Cref{prop1} implies both graphs to be exact, the MAXCUT SDP solution is unique for
\Cref{fig:structure_vs_size}(b) but not for \Cref{fig:structure_vs_size}(a). This is because $\x \x^{\top }$ such that $\x = \mathbbm{1}_{(A,C,E)}-\mathbbm{1}_{(B,D,F)}$ or $\x = \mathbbm{1}_{(A,B,C)}-\mathbbm{1}_{(D,E,F)}$ correspond to distinct optimal rank-1 solutions.
Similarly, \Cref{fig:structure_vs_size}(c) is obtained from
\Cref{fig:structure_vs_size}(b) by deleting the edge $(C,D)$;
this perturbation destroys exactness, as the resulting graph admits
a fractional optimal SDP solution.

Having characterized the case $m_1 = m_2$, we now turn to the setting
$m_1 \neq m_2$. When one part strictly dominates in size, the internal
connectivity of the larger part can influence the preservation of exactness.
The following proposition gives a sufficient condition.

\begin{proposition}\label{claim:2}
Let $|V_A| < |V_B|$. Suppose $\deg_{G[V_B]} (v) \leq \tfrac{1}{2}|V_A|$ for every $v \in V_B$. Then $G'$ is exact and the optimal $\mathrm{MAXCUT\ SDP}$ solution is unique.
\end{proposition}

Unlike \Cref{claim:1}, where the internal structure of $G_A$ and $G_B$ was irrelevant for exactness, \Cref{claim:2} imposes structure on the graph with the larger vertex set. The strength of this condition is that $G_B$ may have an arbitrarily larger vertex set than $G_A$, exactness persists as long as the maximum degree of $G_B$ is bounded by $\frac{1}{2}|V_A|$. One immediate special case is when the larger side is independent, as stated in the following corollary.

\begin{corollary}\label{roc}
$G'$ is exact if $V_B$ is an independent set and $|V_A| \leq |V_B|$. Moreover, the $\mathrm{MAXCUT\ SDP}$ solution is unique. 
\end{corollary}

\ignore{
We now consider a case where internal structure is imposed on both $G_A$ and $G_B$. We show that if both $G_A$ and $G_B$ contain balanced bipartite spanning subgraph, then exactness persists, however the optimal solution is not unique. 

\begin{proposition}\label{claim:3}
    $G'$ is exact when both $G_A$ and $G_B$ contain balanced complete bipartite spanning subgraphs; that is, $G_A$ contains $K(|V_A|/2,\; |V_A|/2)$ and $G_B$ contains $K(|V_B|/2,\;|V_B|/2)$. Moreover, a higher-rank optimal $\mathrm{MAXCUT\ SDP}$ solution always exists for such $G'$. 
\end{proposition}

\begin{corollary}
    $(G_m \vee H_m) \vee (\mathcal{G}_n \vee \mathcal{H}_n)$ is exact for any graphs $G_m, H_m$ on $m$ vertices and $\mathcal{G}_n, \mathcal{H}_n$  on $n$ vertices, respectively. 
\end{corollary}
}

\begin{remark}
    The graph join of exact graphs (including bipartite graphs), however, is not necessarily exact. For example $K_4 \vee P_3$, $P_2 \vee P_3$, and $C_4 \vee P_3$ are not exact.
\end{remark}

We show that one can efficiently verify whether a given graph belongs to either of these classes. Moreover, whenever the corresponding conditions are satisfied, a maximum cut can be computed directly.

\begin{proposition}
The graph classes characterized in \Cref{prop1} and \Cref{claim:2}, as well as the uniqueness of their respective optimal solutions, can be recognized in polynomial time.
\end{proposition}

\begin{proof}
Let $n = |V|$ denote the total number of vertices in the input graph $G$. The recognition algorithm proceeds by constructing the complement graph $G^c$ and computing the sizes of its connected components.

First, we check for the balanced class characterized in \Cref{prop1} ($|V_A| = |V_B|$). This structural case is only possible if $n$ is even. If $n$ is even, we check whether some subcollection of the component sizes of $G^c$ sums exactly to $n/2$. Because the target sum $n/2$ is strictly bounded by $n$, this subset sum instance can be solved via a standard dynamic programming algorithm in strictly polynomial time \cite{MIT6006Lec18}. Furthermore, the uniqueness condition is satisfied if and only if this collection contains a single component of exactly size $n/2$ (meaning either $G_A^c$ or $G_B^c$ is connected), a property verifiable via a linear scan of the component sizes.

For recognizing the graph class characterized in \Cref{claim:2}, we note that we first need to verify that the number of vertices in the largest connected component of $G^c$ is strictly larger than $n/2$. The proof for this claim is as follows. For every $v \in V_B$, $\deg_{G[V_B]}(v) + \deg_{G^c[V_B]}(v) = |V_B| - 1$ which implies that $\deg_{G^c[V_B]}(v) \geq |V_B|- \frac{1}{2}|V_A|- 1$ from the bound $\deg_{G[V_B]} (v) \leq \tfrac{1}{2}|V_A|$. Now, since $|V_B| > |V_A|$, it follows that $ \deg_{G^c[V_B]}(v) > \frac{|V_B|}{2}-1$ for all $v \in V_B$. We claim that this implies that $G^c[V_B]$ is connected. To see this, assume $G^c[V_B]$ is disconnected. Then there exists a connected component $C \subset V_B$ such that $|C| \leq \frac{|V_B|}{2}$ and $\deg_{G^c[C]} (v) \leq |C| -1$ for all $v \in C$. This implies $\deg_{G^c[C]}(v) \leq \frac{|V_B|}{2}-1$ for every $v \in C$, which is a contradiction since $C \subset V_B$. 

Since $|V_B| > n/2$, it follows that $V_B$ is a connected component with size strictly larger than $n/2$ in $G^c$. If no connected component in $G^c$ exceeds this size threshold, the graph is rejected. Otherwise, the unique largest component is identified as $V_B$, and the remaining vertices form $V_A$. Finally, the algorithm verifies the vertex-degree condition $\deg_{G[V_B]}(v) \le \frac{1}{2}|V_A|$ for all $v \in V_B$. Each step can be executed in polynomial time.
\end{proof}

\subsection{Rank of Optimal Solutions in Exact Graphs}\label{secMW}

Thus far, we provided conditions for the existence or absence of higher-rank optimal matrices for certain exact graphs. Nevertheless, the question of when and how these higher-rank optimal solutions occur remains only partially understood. Mirka and Williamson \citep{mirka2024max} highlighted this gap by posing two open problems. We recall the following conjecture from their work: \vspace{1mm}

\textit{If a graph admits a unique partition corresponding to its maximum cut and the GW-relaxation on the graph is exact, then the rank-1 optimal solution must also be the unique optimal solution.}\\

We disprove their conjecture by constructing the following counterexample, on a weighted complete graph: 
\begin{example} \label{counter1}
Consider \(K_4\) weighted by \(\bm=(5,3,4,4)^{\top}\) on vertices \(\{1,2,3,4\}\). The GW-relaxation for this graph is exact (see \Cref{knowngraphs}), since $m_1 + m_2 = m_3 + m_4$. Enumeration shows that the cut \((\{1,2\},\{3,4\})\) is the unique maximum cut, with an optimal value of $64$. The optimal $\mathrm{MAXCUT\ SDP}$ solution, however, is not unique. Consider the following matrix ${\bZ}^*$: 
\begin{align*}
\displaystyle
\bZ^* = \begin{bmatrix}
1 & -\frac{1}{3} & -\frac{1}{2} & -\frac{1}{2}  \\[2mm]
-\frac{1}{3} & 1 & -\frac{1}{6} & -\frac{1}{6}  \\[2mm]
-\frac{1}{2} & -\frac{1}{6} & 1 & -\frac{1}{4}  \\[2mm]
-\frac{1}{2} & -\frac{1}{6} & -\frac{1}{4} & 1 \\ 
\end{bmatrix}
\end{align*}
$\bZ^*$ is a higher-rank feasible primal solution which achieves the same objective value of $64$, and hence is optimal.
\end{example}

Next, we resolve the second question posed in the same work \citep{mirka2024max}:\vspace{1mm}

\textit{Let \( G = (V, E) \) be a graph with multiple maximum cuts. Suppose the corresponding rank-1 reference solutions are optimal primal solutions of the $\mathrm{MAXCUT\ SDP}$. Are all higher-rank optimal solutions necessarily convex combinations of these reference solutions?}\vspace{1mm}

The following example answers the aforementioned question negatively.
\begin{example} \label{counter2}
Consider the complete graph \(K_4\) with vertex set \(\{1,2,3,4\}\), weighted by \(\mathbf{m}=(2,2,4,4)^{\top}\). This graph is exact since $m_1 + m_3 = m_2 + m_4$. The cuts \((\{1,3\},\{2,4\})\) and \((\{1,4\},\{2,3\})\) are the maximum cuts; exhaustive enumeration shows no other partition attains weight \(36\). Let \(\X^*\) and \(\Y^*\) be the corresponding rank-1 optima for the two maximum cuts. Moreover, the following matrix \(\mathbf{\hat{Z}}\) is also feasible and attains the same value~36:
\[
\hat{\bZ} = \begin{bmatrix}
1 & -\frac{1}{5} & -\frac{1}{5} & -\frac{1}{5}  \\[2mm]
-\frac{1}{5} & 1 & -\frac{1}{5} & -\frac{1}{5}  \\[2mm]
-\frac{1}{5} & -\frac{1}{5} & 1 & -\frac{4}{5}  \\[2mm]
-\frac{1}{5} & -\frac{1}{5} & -\frac{4}{5} & 1 \\
\end{bmatrix}
\]
Although \(\hat{\bZ}\) is optimal, it cannot be expressed as a convex combination of \(\X^*\) and \(\Y^*\). This follows from observing that both \( \X^*_{12} = \Y^*_{12} = -1 \), whereas \( \hat{\bZ}_{12} = -\frac{1}{5} \). 
\end{example}

\Cref{counter1} and \Cref{counter2} motivate a closer examination of the structure of higher-rank solutions, particularly in the case of complete graphs. To this end, we introduce the following notion: a vector \(\mathbf{m} \in\mathbb{R}^n_{>0}\) (\(n\ge 3\)) is called \emph{non-dominating} if for every index \(j \in [n]\), $m_j<\sum_{k\ne j} m_k$. The next theorem establishes that, for complete graphs weighted by non-dominating $\mathbf{m}$, the optimal solution to the $\mathrm{MAXCUT \ SDP}$ always has the maximum possible rank $n-1$, independent of the exactness of the relaxation. 
\begin{theorem}\label{thm:3}
Let $G$ be the graph on $n$ vertices weighted by a non-dominating $\bm$. 
Then the $\mathrm{MAXCUT\ SDP}$ admits an optimal solution $\X_G$ with 
\[
\operatorname{rank}(\X_G)=n-1
\quad\text{and}\quad
\dfrac14 \left \langle \L_G, \X_G \right \rangle  =\dfrac14 \left(\sum_{i=1}^n m_i\right)^2.
\]
\end{theorem}

\begin{corollary}\label{convex_hull}
Let $G$ be a graph weighted by a non-dominating $\bm$ admitting at most $n-2$ balanced partitions, i.e., there exist at most \( n-2 \) distinct partitions \( (S_i, S_i^c) \) of \( V \) such that
\( \sum_{j \in S_i} m_j = \sum_{j \notin S_i} m_j \).
Then the $\mathrm{MAXCUT\ SDP}$ admits an optimal solution $\X_G$ that lies outside the convex hull of rank-1 optimal solutions.
\end{corollary}

\ignore{\begin{proofsketch}
For each \(i \in [n]\), construct \(\bu^i\), by setting \(u^i_j = 1\) for \(j \neq i\) and \(u^i_i = -\tfrac{M - m_i}{m_i}\), where \(M = \sum_i m_i\). Observe that each \(\bu^i\) is orthogonal to $\bm$. We show that, there exist \(\{d'_i >0\}, i \in [n]\) such that a weighted sum of the entrywise squares of $\bu^i, i \in [n]$ equals the all-ones vector. Defining \(\X_G = \sum_i d'_i\, \bu^i (\bu^i)^\top\) then yields a feasible point of the $\mathrm{MAXCUT\ SDP}$ whose nullspace is exactly \(\bm\), hence \(\operatorname{rank}(\X_G) = n - 1\). 
The dual variable \(\Y = \tfrac{M}{4}\mathrm{Diag}(\bm)\) produces \(\bS = \tfrac14 \bm \bm^\top\) satisfying complementary slackness since \(\X_G\bm = \bz\).
\end{proofsketch}}

\Cref{thm:3} clarifies the phenomena illustrated in \Cref{counter1} and \Cref{counter2}. We verify that \(\hat{\bZ}\) and \(\bZ^*\) lie in the convex combinations of the optimal rank-1 matrices and their corresponding constructed matrices \(\X_G\) (see \Cref{verification}). These examples reveal that the optimal solutions in the exact relaxation can lie outside the convex hull of the optimal rank-1 solutions. A future direction for research is to fully characterize the conditions under which the optimal $\mathrm{MAXCUT\ SDP}$ solutions in exact relaxations lie outside the convex hull of the rank-1 optima. Motivated by these observations, we next present the following lemmas, which imply the existence of unweighted analogues of \Cref{counter1} and \Cref{counter2} via complete \(n\)-partite graphs, and provides complete characterization of condition for exactness and uniqueness of the $\mathrm{MAXCUT\ SDP}$ solution for complete \(n\)-partite graphs.

\begin{theorem}\label{lem6}
Let \( \widetilde{G} \) be any vertex-split graph of \( G \). Then, \( G \) is exact if and only if \( \widetilde{G} \) is exact.
\end{theorem}

\begin{lemma}\label{splitlemma}
Let $G$ be any graph and suppose there exists a rank-$r$ optimal solution to the $\mathrm{MAXCUT\ SDP}$ for $G$. Then for any vertex-split $\widetilde{G}$ of $G$, there exists a rank-$r$ optimal solution to the $\mathrm{MAXCUT\ SDP}$ for $\widetilde{G}$.
\end{lemma}

\begin{proposition}\label{cor:complete}
$K(a_1,\dots, a_n)$ is exact if and only if one of the following holds:
\begin{enumerate}
    \item (Balanced non-dominating case): There exists a subset $S \subset [n]$ such that $\sum_{i \in S} a_i = \sum_{j \notin S} a_j$. The optimal objective value is $\left(\sum_{i \in S} a_i\right)^2$.
    \item (Dominating case): $a_n \geq \sum_{i \neq n} a_i$. The optimal objective value is $a_n \cdot \sum_{i \neq n} a_i$.
\end{enumerate}
Moreover, the optimal exact $\mathrm{MAX}$-$\mathrm{CUT}$ $\mathrm{SDP}$ solution is unique if and only if $a_n \geq \sum_{i < n} a_i$.
\end{proposition}
\begin{remark}
\Cref{cor:complete} also gives a simple way to construct unweighted exact
instances for which the maximum cut is unique, while the optimal
\(\mathrm{MAXCUT}\) SDP solution is not unique. For example, \(K(1,2,3,4)\) is an exact graph with unique maximum cut partition but has multiple $\mathrm{MAXCUT\ SDP}$ optimal solutions.
\end{remark}

We next examine the behavior of exact graphs under another graph operation, the \emph{lexicographic product}, which features in spectral graph theory \citep{barik2015laplacian} as well as applications such as electric power monitoring \citep{dorbec2008power}. We show that if the first factor $G_1$ of the lexicographic product is exact, then any rank-$r$ optimal solution for $G_1$ can be lifted to a rank-$r$ optimal solution for $G_1 \bullet G_2$; so rank-1 optimal solutions persist.

\begin{theorem}\label{thm:1}
Let $G_1=(V_1,E_1)$ be an exact graph (without isolated vertices) and let $G_2=(V_2,E_2)$. Then the lexicographic product $G_1\!\bullet\! G_2$ is also exact, and its cut value satisfies $\phi(G_1\!\bullet\! G_2) \;=\; |V_2|^2 \,\mc(G_1)$. Moreover, for every optimal rank-$r$ solution of the $\mathrm{MAXCUT\ SDP}$ for an exact graph $G_1$, there exists an optimal rank-$r$ solution for the $\mathrm{MAXCUT\ SDP}$ of $G_1\!\bullet\! G_2$.
\end{theorem}

\ignore{
\begin{proofsketch}
 We extend the optimal cut of $G_1$ by assigning all copies of a vertex of $G_1$ the same label ($+1$ or $-1$) in $G_1 \!\bullet\! G_2$, which yields a cut of value $(|V_2|^2 \cdot \mc(G_1))$. On the dual side, the construction lifts naturally and decomposes into parts that remain feasible, so strong duality certifies optimality. Rank is preserved by taking the Kronecker product of an optimal solution of $G_1$ with the all-ones matrix. 
\end{proofsketch}
}

\begin{corollary}{\label{lex_cor}}
 $G_1 \bullet G_2 \bullet \cdots \bullet G_n$ is exact if $G_1$ is exact.   
\end{corollary}

\begin{remark}
Exactness of other factors ($G_n$ for $n \geq 2$), however, may not suffice: for example, $K_3 \bullet P_3$ yields an SDP optimum of $20.25$, depicting non-exactness (being non-integral) despite $P_3$ being bipartite. While the existence of a rank-r optimal solution for exact $G_1$ extends to $G_1 \bullet G_2$, higher-rank ($> r$) optimal solutions may still appear and hence uniqueness is not preserved. For instance, $P_2$ admits a unique rank-1 optimum. Now, if $G_1 \cong G_2 \cong P_2$, then their product $G_1 \bullet G_2 \cong K_4$, which is exact and also admits higher-rank optima (see \Cref{claim:1}).
\end{remark}

\subsection{Exactness from a structural core} 

Building on the results of the previous sections, which broadened the known families of exact graphs, we now pose a more general question: can these base families serve as building blocks for certifying the exactness of more intricate graphs? If so, suitable decompositions would enable us to
certify the exactness of complex instances by reducing them to well-understood cases.
Motivated by this perspective, we introduce a structural class of graphs, which we call
\textit{split-decomposable graphs}.
This class of graphs provide a concrete example of how more intricate graphs can naturally arise
out of known results, further extending the frontier of exact graphs.
\begin{definition}[Split-decomposable graph]
A weighted graph $G=(V,E,\mathbf{w})$ is called \emph{split-decomposable} if there exists 
a vertex-split graph $\widetilde{G}$ of $G$ with the following properties:
\begin{enumerate}
    \item $\widetilde{G}$ contains a uniformly weighted spanning subgraph $\mathcal{S}'$, where $\mathcal{S}'$ is isomorphic to a vertex-split graph of $\mathcal{S}$, obtained by replacing each vertex of $\mathcal{S}$ by $p \geq 2$ copies. Let $w$ denote the uniform edge weight of $\mathcal{S}$. 
    
    \item The residual edges of $\widetilde{G}$, i.e., the edges in $E(\widetilde{G}) \setminus E(\mathcal{S}')$, are contained entirely within the vertex sets corresponding to the copies of a single vertex of $\mathcal{S}$, and each such edge has weight at most $w/p^{2}$. 
\end{enumerate}
\end{definition}

The definition captures the idea that the structural core of $G$ is governed by a uniform-weight $\mathcal{S}$, referred to as the \textit{skeleton}, which controls the main contribution to the maximum cut, while the residual edges are sufficiently light (bounded by $w/p^{2}$) to preserve exactness. Note that, by definition, $\widetilde{G}$ may be $G$ itself. The central consequence of this setup, established in the next theorem, is that the exactness of any skeleton $\mathcal{S}$ guarantees the exactness of the entire graph $G$. 

\begin{proposition}\label{thm:splitexact}
If a split-decomposable graph $G$ has a skeleton $\mathcal S$ such that $\mathcal S$ is exact, then $G$ is exact.    
\end{proposition}

\begin{figure}   
\begin{subfigure}[b]{0.4\textwidth}
\centering
\scalebox{0.7}{%
\begin{tikzpicture}
  \begin{scope}[every node/.style={circle,draw,inner sep=1pt}]
    \node (A1) at (3.0000, 0.0000) {$A$};
    \node (A2) at (2.656371, 1.394168) {$B$};
    \node (A3) at (1.701805, 2.470597) {$C$};
    \node (B1) at (0.330667, 2.981721) {$D$};
    \node (B2) at (-1.114987, 2.785099) {$E$};
    \node (B3) at (-2.345492, 1.870471) {$F$};
    \node (C1) at (-2.924785, 0.667563) {$G$};
    \node (C2) at (-2.924785, -0.667563) {$H$};
    \node (C3) at (-2.345492, -1.870471) {$I$};
    \node (D1) at (-1.114987, -2.785099) {$J$};
    \node (D2) at (0.330667, -2.981721) {$K$};
    \node (D3) at (1.701805, -2.470597) {$L$};
    \node (E) at (2.656371, -1.394168) {$M$};
  \end{scope}
  \begin{scope}[>={Stealth[black]},
                  every node/.style={},
                  every edge/.style={draw=black, thick}]
        \path [-] (A1) edge[cedge=cyan] node {} (B1);
        \path [-] (A1) edge[cedge=cyan] node {} (B2);
        \path [-] (A1) edge[cedge=cyan] node {} (B3);
        \path [-] (A1) edge[cedge=cyan] node {} (C1);
        \path [-] (A1) edge[cedge=cyan] node {} (C2);
        \path [-] (A1) edge[cedge=cyan] node {} (C3);
        \path [-] (A1) edge[cedge=cyan] node {} (D1);
        \path [-] (A1) edge[cedge=cyan] node {} (D2);
        \path [-] (A1) edge[cedge=cyan] node {} (D3);
        \path [-] (A1) edge[cedge=green!40!gray] node {} (E);
        \path [-] (A2) edge[cedge=red!40!gray] node {} (A3);
        \path [-] (A2) edge[cedge=cyan] node {} (B1);
        \path [-] (A2) edge[cedge=cyan] node {} (B2);
        \path [-] (A2) edge[cedge=cyan] node {} (B3);
        \path [-] (A2) edge[cedge=cyan] node {} (C1);
        \path [-] (A2) edge[cedge=cyan] node {} (C2);
        \path [-] (A2) edge[cedge=cyan] node {} (C3);
        \path [-] (A2) edge[cedge=cyan] node {} (D1);
        \path [-] (A2) edge[cedge=cyan] node {} (D2);
        \path [-] (A2) edge[cedge=cyan] node {} (D3);
        \path [-] (A2) edge[cedge=green!40!gray] node {} (E);
        \path [-] (A3) edge[cedge=cyan] node {} (B1);
        \path [-] (A3) edge[cedge=cyan] node {} (B2);
        \path [-] (A3) edge[cedge=cyan] node {} (B3);
        \path [-] (A3) edge[cedge=cyan] node {} (C1);
        \path [-] (A3) edge[cedge=cyan] node {} (C2);
        \path [-] (A3) edge[cedge=cyan] node {} (C3);
        \path [-] (A3) edge[cedge=cyan] node {} (D1);
        \path [-] (A3) edge[cedge=cyan] node {} (D2);
        \path [-] (A3) edge[cedge=cyan] node {} (D3);
        \path [-] (A3) edge[cedge=green!40!gray] node {} (E);
        \path [-] (B1) edge[cedge=red!40!gray] node {} (B2);
        \path [-] (B1) edge[cedge=cyan] node {} (C1);
        \path [-] (B1) edge[cedge=cyan] node {} (C2);
        \path [-] (B1) edge[cedge=cyan] node {} (C3);
        \path [-] (B1) edge[cedge=cyan] node {} (D1);
        \path [-] (B1) edge[cedge=cyan] node {} (D2);
        \path [-] (B1) edge[cedge=cyan] node {} (D3);
        \path [-] (B1) edge[cedge=green!40!gray] node {} (E);
        \path [-] (B2) edge[cedge=red!40!gray] node {} (B3);
        \path [-] (B2) edge[cedge=cyan] node {} (C1);
        \path [-] (B2) edge[cedge=cyan] node {} (C2);
        \path [-] (B2) edge[cedge=cyan] node {} (C3);
        \path [-] (B2) edge[cedge=cyan] node {} (D1);
        \path [-] (B2) edge[cedge=cyan] node {} (D2);
        \path [-] (B2) edge[cedge=cyan] node {} (D3);
        \path [-] (D1) edge[cedge=cyan] node {} (D2);
        \path [-] (B2) edge[cedge=green!40!gray] node {} (E);
        \path [-] (B3) edge[cedge=cyan] node {} (C1);
        \path [-] (B3) edge[cedge=cyan] node {} (C2);
        \path [-] (B3) edge[cedge=cyan] node {} (C3);
        \path [-] (B3) edge[cedge=cyan] node {} (D1);
        \path [-] (B3) edge[cedge=cyan] node {} (D2);
        \path [-] (B3) edge[cedge=cyan] node {} (D3);
        \path [-] (B3) edge[cedge=green!40!gray] node {} (E);
        \path [-] (C1) edge[cedge=red!40!gray] node {} (C2);
        \path [-] (C1) edge[cedge=red!40!gray] node {} (C3);
        \path [-] (C2) edge[cedge=red!40!gray] node {} (C3);
        \path [-] (D1) edge[cedge=red!40!gray] node {} (D3);
        \path [-] (D2) edge[cedge=red!40!gray] node {} (D3);
        \path [-] (C1) edge[cedge=green!40!gray] node {} (E);
    \end{scope}
    \end{tikzpicture}}
    \caption*{Labeled Graph $\mathcal{G}$}
\end{subfigure}
\tikz[baseline=-\baselineskip]\draw[thick,implies-implies,double equal sign distance] (0,2) -- (0.5,2);\qquad
\begin{subfigure}[b]{0.4\textwidth}
\scalebox{0.7}{%
\begin{tikzpicture}
  \begin{scope}[every node/.style={circle,draw,inner sep=1pt}]
      \node (A1) at (3.0000, 0.0000) {$A$};
      \node (A2) at (2.7362, 1.2202) {$B$};
      \node (A3) at (2.0090, 2.2270) {$C$};
      \node (B1) at (0.9271, 2.8529) {$D$};
      \node (B2) at (-0.4680, 2.9631) {$E$};
      \node (B3) at (-1.5000, 2.5981) {$F$};
      \node (C1) at (-2.4271, 1.7634) {$G$};
      \node (C2) at (-2.9271, 0.6245) {$H$};
      \node (C3) at (-2.9271, -0.6245) {$I$};
      \node (D1) at (-2.4271, -1.7634) {$J$};
      \node (D2) at (-1.5000, -2.5981) {$K$};
      \node (D3) at (-0.4680, -2.9631) {$L$};
      \node (M1) at (0.9271, -2.8529) {$M_1$};
      \node (M2) at (2.0090, -2.2270) {$M_2$};
      \node (M3) at (2.7362, -1.2202) {$M_3$};
  \end{scope}
  \begin{scope}[>={Stealth[black]},
                  every node/.style={},
                  every edge/.style={draw=black, thick}]
        \path [-] (A1) edge[cedge=cyan] node {} (B1);
        \path [-] (A1) edge[cedge=cyan] node {} (B2);
        \path [-] (A1) edge[cedge=cyan] node {} (B3);
        \path [-] (A1) edge[cedge=cyan] node {} (C1);
        \path [-] (A1) edge[cedge=cyan] node {} (C2);
        \path [-] (A1) edge[cedge=cyan] node {} (C3);
        \path [-] (A1) edge[cedge=cyan] node {} (D1);
        \path [-] (A1) edge[cedge=cyan] node {} (D2);
        \path [-] (A1) edge[cedge=cyan] node {} (D3);
        \path [-] (A1) edge[cedge=cyan] node {} (M1);
        \path [-] (A1) edge[cedge=cyan] node {} (M2);
        \path [-] (A1) edge[cedge=cyan] node {} (M3);
        \path [-] (A2) edge[cedge=red!40!gray] node {} (A3);
        \path [-] (A2) edge[cedge=cyan] node {} (B1);
        \path [-] (A2) edge[cedge=cyan] node {} (B2);
        \path [-] (A2) edge[cedge=cyan] node {} (B3);
        \path [-] (A2) edge[cedge=cyan] node {} (C1);
        \path [-] (A2) edge[cedge=cyan] node {} (C2);
        \path [-] (A2) edge[cedge=cyan] node {} (C3);
        \path [-] (A2) edge[cedge=cyan] node {} (D1);
        \path [-] (A2) edge[cedge=cyan] node {} (D2);
        \path [-] (A2) edge[cedge=cyan] node {} (D3);
        \path [-] (A2) edge[cedge=cyan] node {} (M1);
        \path [-] (A2) edge[cedge=cyan] node {} (M2);
        \path [-] (A2) edge[cedge=cyan] node {} (M3);
        \path [-] (A3) edge[cedge=cyan] node {} (B1);
        \path [-] (A3) edge[cedge=cyan] node {} (B2);
        \path [-] (A3) edge[cedge=cyan] node {} (B3);
        \path [-] (A3) edge[cedge=cyan] node {} (C1);
        \path [-] (A3) edge[cedge=cyan] node {} (C2);
        \path [-] (A3) edge[cedge=cyan] node {} (C3);
        \path [-] (A3) edge[cedge=cyan] node {} (D1);
        \path [-] (A3) edge[cedge=cyan] node {} (D2);
        \path [-] (A3) edge[cedge=cyan] node {} (D3);
        \path [-] (A3) edge[cedge=cyan] node {} (M1);
        \path [-] (A3) edge[cedge=cyan] node {} (M2);
        \path [-] (A3) edge[cedge=cyan] node {} (M3);
        \path [-] (B1) edge[cedge=red!40!gray] node {} (B2);
        \path [-] (B1) edge[cedge=cyan] node {} (C1);
        \path [-] (B1) edge[cedge=cyan] node {} (C2);
        \path [-] (B1) edge[cedge=cyan] node {} (C3);
        \path [-] (B1) edge[cedge=cyan] node {} (D1);
        \path [-] (B1) edge[cedge=cyan] node {} (D2);
        \path [-] (B1) edge[cedge=cyan] node {} (D3);
        \path [-] (B1) edge[cedge=cyan] node {} (M1);
        \path [-] (B1) edge[cedge=cyan] node {} (M2);
        \path [-] (B1) edge[cedge=cyan] node {} (M3);
        \path [-] (B2) edge[cedge=red!40!gray] node {} (B3);
        \path [-] (B2) edge[cedge=cyan] node {} (C1);
        \path [-] (B2) edge[cedge=cyan] node {} (C2);
        \path [-] (B2) edge[cedge=cyan] node {} (C3);
        \path [-] (B2) edge[cedge=cyan] node {} (D1);
        \path [-] (B2) edge[cedge=cyan] node {} (D2);
        \path [-] (B2) edge[cedge=cyan] node {} (D3);
        \path [-] (D1) edge[cedge=cyan] node {} (D2);
        \path [-] (B2) edge[cedge=cyan] node {} (M1);
        \path [-] (B2) edge[cedge=cyan] node {} (M2);
        \path [-] (B2) edge[cedge=cyan] node {} (M3);
        \path [-] (B3) edge[cedge=cyan] node {} (C1);
        \path [-] (B3) edge[cedge=cyan] node {} (C2);
        \path [-] (B3) edge[cedge=cyan] node {} (C3);
        \path [-] (B3) edge[cedge=cyan] node {} (D1);
        \path [-] (B3) edge[cedge=cyan] node {} (D2);
        \path [-] (B3) edge[cedge=cyan] node {} (D3);
        \path [-] (B3) edge[cedge=cyan] node {} (M1);
        \path [-] (B3) edge[cedge=cyan] node {} (M2);
        \path [-] (B3) edge[cedge=cyan] node {} (M3);
        \path [-] (C1) edge[cedge=red!40!gray] node {} (C2);
        \path [-] (C1) edge[cedge=red!40!gray] node {} (C3);
        \path [-] (C2) edge[cedge=red!40!gray] node {} (C3);
        \path [-] (D1) edge[cedge=red!40!gray] node {} (D3);
        \path [-] (D2) edge[cedge=red!40!gray] node {} (D3);
        \path [-] (C1) edge[cedge=cyan] node {} (M1);
        \path [-] (C1) edge[cedge=cyan] node {} (M2);
        \path [-] (C1) edge[cedge=cyan] node {} (M3);
    \end{scope}
    \end{tikzpicture}}
    \caption*{\hspace*{-2cm}$\widetilde{\mathcal{G}}$}
\end{subfigure}
\vskip0.8\baselineskip
\noindent\hspace*{-2cm}\tikz[baseline=-\baselineskip]\draw[thick,implies-implies,double equal sign distance, transform canvas={xshift=-4.25cm}] (0,0.5) -- (0,0);
\vskip0.01\baselineskip
\begin{subfigure}{0.4\textwidth}
\centering
\scalebox{0.7}{%
\begin{tikzpicture}
            \begin{scope}[every node/.style={ellipse,draw}]
              \node (A) at (0,-0.9) {$A,B,C$};
              \node (D) at (5.3,-0.9) {$D,E,F$};
              \node (G) at (2.7,4) {$G,H,I$} ;
              \node (J) at (2.7,2) {$J,K,L$} ;
              \node (M) at (2.7,0) {$M_1,M_2,M_3$} ;
            \end{scope}
            \begin{scope}[>={Stealth[black]}, every node/.style={}, every edge/.style={draw=black, very thick}]
              \path[-] (A) edge (D);
              \path[-] (M) edge (D);
              \path[-] (D) edge (J);
              \path[-] (D) edge (G);
              \path[-] (M) edge (A);
              \path[-] (G) edge (A);
              \path[-] (J) edge (A);
            \end{scope}
          \end{tikzpicture}%
        }
        \caption*{Skeleton $\mathcal{S}$}
      \end{subfigure}    
\tikz[baseline=-\baselineskip]\draw[thick,implies-implies,double equal sign distance] (0,2) -- (0.5,2);\qquad
\begin{subfigure}[b]{0.4\textwidth}
\scalebox{0.7}{%
\begin{tikzpicture}
  \begin{scope}[every node/.style={circle,draw,inner sep=1pt}]
      \node (A1) at (3.0000, 0.0000) {$A$};
      \node (A2) at (2.7362, 1.2202) {$B$};
      \node (A3) at (2.0090, 2.2270) {$C$};
      \node (B1) at (0.9271, 2.8529) {$D$};
      \node (B2) at (-0.4680, 2.9631) {$E$};
      \node (B3) at (-1.5000, 2.5981) {$F$};
      \node (C1) at (-2.4271, 1.7634) {$G$};
      \node (C2) at (-2.9271, 0.6245) {$H$};
      \node (C3) at (-2.9271, -0.6245) {$I$};
      \node (D1) at (-2.4271, -1.7634) {$J$};
      \node (D2) at (-1.5000, -2.5981) {$K$};
      \node (D3) at (-0.4680, -2.9631) {$L$};
      \node (M1) at (0.9271, -2.8529) {$M_1$};
      \node (M2) at (2.0090, -2.2270) {$M_2$};
      \node (M3) at (2.7362, -1.2202) {$M_3$};
  \end{scope}
  \begin{scope}[>={Stealth[black]},
                  every node/.style={},
                  every edge/.style={draw=black, thick}]
        \path [-] (A1) edge[cedge=cyan] node {} (B1);
        \path [-] (A1) edge[cedge=cyan] node {} (B2);
        \path [-] (A1) edge[cedge=cyan] node {} (B3);
        \path [-] (A1) edge[cedge=cyan] node {} (C1);
        \path [-] (A1) edge[cedge=cyan] node {} (C2);
        \path [-] (A1) edge[cedge=cyan] node {} (C3);
        \path [-] (A1) edge[cedge=cyan] node {} (D1);
        \path [-] (A1) edge[cedge=cyan] node {} (D2);
        \path [-] (A1) edge[cedge=cyan] node {} (D3);
        \path [-] (A1) edge[cedge=cyan] node {} (M1);
        \path [-] (A1) edge[cedge=cyan] node {} (M2);
        \path [-] (A1) edge[cedge=cyan] node {} (M3);
        \path [-] (A2) edge[cedge=cyan] node {} (B1);
        \path [-] (A2) edge[cedge=cyan] node {} (B2);
        \path [-] (A2) edge[cedge=cyan] node {} (B3);
        \path [-] (A2) edge[cedge=cyan] node {} (C1);
        \path [-] (A2) edge[cedge=cyan] node {} (C2);
        \path [-] (A2) edge[cedge=cyan] node {} (C3);
        \path [-] (A2) edge[cedge=cyan] node {} (D1);
        \path [-] (A2) edge[cedge=cyan] node {} (D2);
        \path [-] (A2) edge[cedge=cyan] node {} (D3);
        \path [-] (A2) edge[cedge=cyan] node {} (M1);
        \path [-] (A2) edge[cedge=cyan] node {} (M2);
        \path [-] (A2) edge[cedge=cyan] node {} (M3);
        \path [-] (A3) edge[cedge=cyan] node {} (B1);
        \path [-] (A3) edge[cedge=cyan] node {} (B2);
        \path [-] (A3) edge[cedge=cyan] node {} (B3);
        \path [-] (A3) edge[cedge=cyan] node {} (C1);
        \path [-] (A3) edge[cedge=cyan] node {} (C2);
        \path [-] (A3) edge[cedge=cyan] node {} (C3);
        \path [-] (A3) edge[cedge=cyan] node {} (D1);
        \path [-] (A3) edge[cedge=cyan] node {} (D2);
        \path [-] (A3) edge[cedge=cyan] node {} (D3);
        \path [-] (A3) edge[cedge=cyan] node {} (M1);
        \path [-] (A3) edge[cedge=cyan] node {} (M2);
        \path [-] (A3) edge[cedge=cyan] node {} (M3);
        \path [-] (B1) edge[cedge=cyan] node {} (C1);
        \path [-] (B1) edge[cedge=cyan] node {} (C2);
        \path [-] (B1) edge[cedge=cyan] node {} (C3);
        \path [-] (B1) edge[cedge=cyan] node {} (D1);
        \path [-] (B1) edge[cedge=cyan] node {} (D2);
        \path [-] (B1) edge[cedge=cyan] node {} (D3);
        \path [-] (B1) edge[cedge=cyan] node {} (M1);
        \path [-] (B1) edge[cedge=cyan] node {} (M2);
        \path [-] (B1) edge[cedge=cyan] node {} (M3);
        \path [-] (B2) edge[cedge=cyan] node {} (C1);
        \path [-] (B2) edge[cedge=cyan] node {} (C2);
        \path [-] (B2) edge[cedge=cyan] node {} (C3);
        \path [-] (B2) edge[cedge=cyan] node {} (D1);
        \path [-] (B2) edge[cedge=cyan] node {} (D2);
        \path [-] (B2) edge[cedge=cyan] node {} (D3);
        \path [-] (B2) edge[cedge=cyan] node {} (M1);
        \path [-] (B2) edge[cedge=cyan] node {} (M2);
        \path [-] (B2) edge[cedge=cyan] node {} (M3);
        \path [-] (B3) edge[cedge=cyan] node {} (C1);
        \path [-] (B3) edge[cedge=cyan] node {} (C2);
        \path [-] (B3) edge[cedge=cyan] node {} (C3);
        \path [-] (B3) edge[cedge=cyan] node {} (D1);
        \path [-] (B3) edge[cedge=cyan] node {} (D2);
        \path [-] (B3) edge[cedge=cyan] node {} (D3);
        \path [-] (B3) edge[cedge=cyan] node {} (M1);
        \path [-] (B3) edge[cedge=cyan] node {} (M2);
        \path [-] (B3) edge[cedge=cyan] node {} (M3);
        \path [-] (C1) edge[cedge=cyan] node {} (M1);
        \path [-] (C1) edge[cedge=cyan] node {} (M2);
        \path [-] (C1) edge[cedge=cyan] node {} (M3);
    \end{scope}
    \end{tikzpicture}}
    \caption*{\hspace*{-2cm}Spanning subgraph $\mathcal{S'}$}
\end{subfigure}
\caption{\centering Decomposition of $\mathcal{G}$. Edge colors indicate weights: \textcolor{red!40!gray}{Red} (1), \textcolor{cyan}{Blue} (2), \textcolor{green!40!gray}{Green} (6) \& \textcolor{black}{Black} (18).}
    \label{fig:labelledexample}
\end{figure}

As an illustration, we revisit \Cref{example1} and verify that it is split-decomposable with exact skeleton. Consider the vertex labeling of the graph of \Cref{example1}, as shown in \Cref{fig:labelledexample}. First, we perform a split on the vertex $M$ into three distinct vertices, $M_1, M_2, M_3$, to obtain $\mathcal{\widetilde{G}}$. Within $\mathcal{\widetilde{G}}$, observe the existence of a spanning subgraph $\mathcal{S'}$ with uniform edge weights. This subgraph $\mathcal{S'}$ is isomorphic to a vertex-split graph of $\mathcal{S}$, as shown in \Cref{fig:labelledexample}. $\mathcal{S}$ is isomorphic to uniformly weighted $K(1,1,3)$ and hence, is exact by \Cref{claim:2}. Furthermore, the edges in $E(\mathcal{\widetilde{G}}) \setminus E(\mathcal{S'})$ have maximum weight $w_{max} = w_{JK} = 2 \leq \frac{18}{3^2}$. Hence, $\mathcal{G}$ is a split-decomposable graph with an exact skeleton, and by \Cref{thm:splitexact} we conclude that $\mathcal{G}$ is exact.

\begin{remark}
Through \Cref{example1}, we wish to highlight that provided a decomposition of a graph $G$, the factorization can be utilized to sufficiently identify exactness of $G$. However, we do not propose a general algorithm for obtaining such a decomposition. The existence of such an algorithm which efficiently decomposes a given graph $G$ to determine its exactness remains an open direction for future research. 
\end{remark}

\section{Proofs}\label{sec5}
In this section, we present the proofs of the results discussed in this work. We begin with introducing some additional notation necessary for the proofs, and highlight a few key results that are invoked during the discussion of the proofs. 
\subsection{Additional Notation}
Given $\A \in \mathbb{R}^{n \times n}$, we use $\A \succeq \mathbf{0}$ $(\A \succ \mathbf{0})$ to denote positive semidefiniteness (positive definiteness) of $\A$. Through $\mathcal{N}(\A)$, $\range(\A)$, $\tr(\A)$ , $\lambda_{\max}(\A)$ and $\lambda_{\min}(\A)$, we denote the nullspace, range, trace, maximum eigenvalue and minimum eigenvalue, respectively for a matrix $\A$. For $\X\in\mathbb{R}^{n\times n}$ write $\diag(X):=(X_{11},\dots,X_{nn})^\top$. $\I_n$ represents the $n \times n$ identity matrix, $\1_m$ represents the $\bm$-dimensional all-ones vector and $\1_{m \times n}$ represents the $m \times n$ all-ones matrix. Kronecker product of two matrices $\A$ and $\B$ is denoted by $\A\otimes \B$. Given any matrix $\mathbf{B}$, let $\mathbf{B}^+$ be its Moore-Penrose pseudo-inverse. Given a nonzero vector \( \bv \in \mathbb{R}^n \), we denote by $\mathbf{v}^\perp = \{ \x \in \mathbb{R}^n : \x^{\top} \mathbf{v}  = 0 \}$ the orthogonal complement of the span of \( \mathbf{v} \).

\subsection{Some useful results}
\begin{enumerate}
\setlength\itemsep{1em}
\item  If $\mathbf{A}, \mathbf{B}$ are symmetric matrices, $\mathbf{A} - \mathbf{B} \succeq \bz$ if $\lambda_{\min} (\mathbf{A}) - \lambda_{\max} (\mathbf{B}) \geq \bz$.
 \item Consider a symmetric matrix $\mathbf{M}$ of the form $\mathbf{M} =$ $\begin{bmatrix}
\mathbf{A} & \mathbf{C} \\
\mathbf{C}^{\top} & \mathbf{B}
\end{bmatrix}$, where $\mathbf{A}$ and $\mathbf{B}$ are symmetric matrices of appropriate dimensions. The following results hold (see  \citep{gallier2020schur} for details):
    \begin{enumerate}
        \item If $\mathbf{B} \succ \bz$, then $\mathbf{M} \succeq \bz$ iff $\mathbf{A} - \mathbf{C} \mathbf{B}^{-1} \mathbf{C}^{\top} \succeq \bz$.
        \item If $\mathbf{B} \succeq \bz$, $\mathbf{M} \succeq \mathbf{0}$ iff $\range(\mathbf{C}^{\top}) \subseteq \range(\mathbf{B})$ and $\mathbf{A} - \mathbf{C} \mathbf{B}^{+} \mathbf{C}^{\top} \succeq \mathbf{0}$.
        \item (\cite[Theorem 1]{carlson1974generalization}) \textit{(Rank inequality for the pseudo-Schur complement)} $\rank(\mathbf{M}) \geq \rank(\mathbf{B}) + \rank(\mathbf{A} - \mathbf{C} \mathbf{B}^{+} \mathbf{C}^{\top})$. 
    \end{enumerate} 
\item
\begin{enumerate}
    \item $\rank(\A \otimes \B) =  \rank(\A)\rank(\B)$. 
    \item For any positive semidefinite matrices $\mathbf{A}, \mathbf{B}$: $\lambda_{\max} (\mathbf{A} \otimes \mathbf{B}) = \lambda_{\max} (\mathbf{A}) \times \lambda_{\max} (\mathbf{B})$ and $\lambda_{\min} (\mathbf{A} \otimes \mathbf{B}) = \lambda_{\min} (\mathbf{A}) \times \lambda_{\min} (\mathbf{B})$.
\end{enumerate}

\item Given subspaces $U, V \subseteq \mathbb{R}^n$, the \emph{Grassmann identity} states that
\[
\dim(U+V) = \dim U + \dim V - \dim(U \cap V)
\]
where $U + V := \{\, u + v \mid u \in U,\, v \in V \,\}$.

\item Following are some useful properties of Laplacian of a graph:

\begin{enumerate}
        \item For an unweighted graph \(G\) with \(|V|=n\) we have \(\lambda_{\min}(\L_G)=0\) and \(\lambda_{\max}(\L_G)\le n\). Moreover
\[
\lambda_{\max}(\L_G)\le\max_{(u,v)\in E}(\deg_G(u)+\deg_G(v)),
\]
and this bound is tight iff \(G\) is bipartite regular.
    \item \textit{(Laplacian complement identity)}. Let the Laplacian eigenvalues of $G$ be $0 = \lambda_1 \le \lambda_2 \le \cdots \le \lambda_{\max}$. Then, the Laplacian eigenvalues of its complement $G^c$ are given by $0, \; n - \lambda_{\max}, \; n - \lambda_{n-1}, \; \dots, \; n - \lambda_2$.
    \item $\dim \mathcal{N}(\L_G)=$ number of connected components of $G$.
\end{enumerate}
\end{enumerate}

\subsection{Rank Identity Lemma}

The following matrix rank identity will be used in the subsequent proofs:

\begin{lemma}\label{sylvester}
   For any optimal primal solution $(\X^*)$ for $\mathrm{MAXCUT\ SDP}$ and corresponding optimal dual solution $(\bS^*)$ for $\mathrm{MAXCUT\ DUAL\ SDP}$, the following equality holds:
   $\rank(\X^*) + \rank(\bS^*) = n - \dim \left( \mathcal{N}(\X^*) \cap \mathcal{N}(\bS^*)\right).$
\end{lemma}

\begin{proof}
    By strong duality, any optimal primal--dual pair $(\X^*,\bS^*)$ satisfies 
$\X^*\bS^*=\bS^*\X^*=\bz$. Hence $\range(\X^*)\subseteq \mathcal{N}(\bS^*)$ and 
$\range(\bS^*)\subseteq \mathcal{N}(\X^*)$, which implies that 
$\mathcal{N}(\X^*)+\mathcal{N}(\bS^*)=\R^n$. Using Grassman identity, we have
\[
n = \dim \mathcal{N}(\X^*) + \dim \mathcal{N}(\bS^*) 
    - \dim \bigl(\mathcal{N}(\X^*) \, \cap \, \mathcal{N}(\bS^*)\bigr).
\]
The desired equality then follows from the rank-nullity theorem.
\end{proof}

To establish exactness for the graph classes considered in this manuscript, we construct feasible primal and dual solutions that attain the same objective value, thereby certifying optimality. From the reference optimal rank-1 primal solution, a maximum cut can be explicitly extracted for each graph family. 

\subsection{Proof of \Cref{claim:1}}\label{proof1}
Let $m_1 = m_2 = n$. Observe that the graphs $ G_K,G_{A},G_{B} $  have no edges in common. The Laplacian matrix of the union graph $G'$ is therefore, simply the sum of the individual Laplacians: $\L_{G'} = \L_{G_K} + \L_{G_{A}} + \L_{G_{B}}$. Since $ G_K $ is an $ n-$ regular bipartite graph, $ \lambda_{\max}(\L_{G_K}) = 2n $ and it follows that  
\begin{equation*}
 \lambda_{\max}(\L_{G'}) \geq \lambda_{\max}(\L_{G_K}) = 2n   
\end{equation*}
 Since $\lambda_{max}$ of the Laplacian of any graph is upper bounded by the number of vertices, we conclude that $ \lambda_{\max}(\L_{G'}) = 2n $. Now, take $\x \in \R^{|V_K|}$ such that $\x = \ind_{V_A} - \ind_{V_B}$. Then $z_{\mathsf{\mc}}$ at $ \X^* = \x\x^{\top} $ is $ n^2 $. We obtain the same value for MAXCUT DUAL by taking $ Y = \frac{n}{2} \I_{2n} $. Furthermore, the dual slack matrix $\bS^* = \Y - \frac{1}{4} \L_{G'}$ is positive semidefinite, since: 
 
\begin{equation}\label{dual_feasibilty}
\lambda_{\min} (\Y) - \lambda_{\max} (\frac{1}{4}\L_{G'}) = 0.    
\end{equation}

Thus, the dual constraint is satisfied, and by strong duality $ \X^* = \x\x{^{\top}} $ is an optimal solution.  

To prove uniqueness, first observe that the dual optimal matrix admits the block form
\[
\bS^* = \begin{pmatrix}
\mathbf{A} & \frac{1}{4}\1_{n \times n} \\[3pt]
\frac{1}{4}\1_{n \times n} & \mathbf{B}
\end{pmatrix},\]

$\text{where} \; \; 
\mathbf{A}=\tfrac14(n \I_n-\L(G_A)), \; \text{and} \;
\mathbf{B}=\tfrac14(n\I_n-\L(G_B))$.\\

From the optimality of $(\X^*, \bS^*)$, it follows by \Cref{sylvester} that

\begin{equation}\label{sylv}
\rank(\X^*)+\rank(\bS^*) \;=\; 2n - k_0, 
\end{equation}
$\text{where}\; k_0:= \dim \left(\mathcal{N}(\X^*) \cap \mathcal{N}(\bS^*)\right)$.\\

On the other hand, rank inequality for the pseudo-Schur complement yields
\[
\rank(\bS^*) \;\geq\; \rank(\B)+\rank \left(\A-\frac{1}{16}\1_{n \times n}\B^+\1_{n \times n}\right).
\]
Since $\1_n$ is a non-zero eigenvector of $\B$; $\1_{n \times n}\B^+\1_{n \times n}$ is a rank-1 matrix and we deduce
\[
\rank(\bS^*) \;\ge\; (n-k_B)+(n-k_A)-1 \;=\; 2n-(k_A+k_B)-1.
\]
where $k_A=\dim\mathcal{N}(\A)$ and $k_B=\dim\mathcal{N}(\B)$.
Combining with \Cref{sylv} gives
\[
\rank(\X^*) \;\le\; (k_A+k_B)-k_0+1.
\]
Now, observe that $\1_n$ is a nonzero eigenvector of $\mathbf{A}$. 
Hence, every vector $\bv \in \mathcal{N}(\mathbf{A})$ is orthogonal to $\1_n$, as eigenvectors of a symmetric matrix 
associated with distinct eigenvalues are orthogonal. This implies that all null-space directions of $\A$ (and similarly 
of $\B$) lie entirely in the subspace orthogonal to $\1_n$. Consequently, any $\bv \in \mathcal{N}(\A)$ or $\bw \in \mathcal{N}(\B)$ extends naturally to a vector 
$(\bv^{\top},\mathbf{0}^{\top})^{\top} \in \R^{2n}$ or $(\mathbf{0}^{\top},\bw^{\top})^{\top} \in \R^{2n}$, respectively that is annihilated by both $\bS^*$ and $\X^*$. 

Hence
$k_0 \ge \max\{k_A,k_B\}$ which implies
\[
\rank(\X^*) \;\leq\; 1+\min\{k_A,k_B\}.
\]

Now $\A$ is singular if and only if $\lambda_{\max}(\L(G_A))=n$, which is equivalent to $\lambda_2(\L(G_A^c))=0$, i.e.\ $G_A^c$ is disconnected. Thus $\A$ is nonsingular (so $k_A=0$)
iff $G_A^c$ is connected, and similarly for $\B$ and $G_B^c$. The assumption that at least one of $G_A^c,G_B^c$
is connected implies $\min\{k_A,k_B\}=0$. It follows that $\rank(\X^*) = 1$ as the zero matrix is not feasible. If there were two distinct rank-1 optima, their convex combination would have higher-rank, contradicting the bound. Thus the optimal  solution must be unique. The proof of the necessity direction relies on structural properties of complete-$k$-partite graphs and is deferred to \Cref{onlyif}.\myQED

\subsubsection{Proof of \Cref{claim:3}}
    Exactness can be verified immediately as an application of \Cref{prop1}. For non-uniqueness observe that \( \x = (\x_A^\top, \x_B^\top)^\top \), where \( \x_A \in \{\pm 1\}^{|V_A|} \) and  $\x_B \in  \{\pm 1\}^{|V_B|}$  are the cut-sign vectors corresponding to the optimal Max-Cut partitions of \( G_A \) and \( G_B \), respectively is an optimal solution; flipping the signs in one block also gives an optimal solution, as in $(-\x_A^{\top}, \x_B^{\top})^{\top}$. \myQED

\ignore{
\begin{proof}
  Since $ G_K,G_{A},G_{B} $  have no edges in common, we have $ \lambda_{\max}(\L_{G'}) = m_1+m_2 $. Let \( \x = (\x_A^\top, \x_B^\top)^\top \), where \( \x_A \in \{\pm 1\}^{|V_A|} \) and  $\x_B \in  \{\pm 1\}^{|V_B|}$  are the cut-sign vectors corresponding to the optimal Max-Cut partitions of \( G_A \) and \( G_B \), respectively.This gives 

    \begin{align*}
\x^\top \L_{G'} \x 
  &= \x^\top \L_{G_K} \x 
   + \x_A^\top \L_{G_{A}} \x_A 
   + \x_B^\top \L_{G_{B}} \x_B \\
  &= \frac{m_1 m_2}{2} + \frac{m_1^2}{4} + \frac{m_2^2}{4} \\
  &= \frac{(m_1+m_2)^2}{4}
\end{align*}

For constructing a dual with same objective value, take $\Y = \frac{m_1+m_2}{4} \times \I_{m_1 + m_2}$. This dual can be verified to be feasible because $\lambda_{\min}(\Y) - \frac{1}{4}\lambda_{\max}(\L_{G'}) \geq 0$. \\

Closer inspection shows that at least two distinct partitions attain the same optimal value 
\(\frac{(m_1 + m_2)^2}{4}\): flipping the sign of all entries in one block, say \(\x_A\), leaves the cut within \(G_A\) unchanged due to its balanced structure, while the complete bipartite form of \(G_K\) together with the balancedness of \(G_B\) ensures that the total cut across \(V_A\) and \(V_B\) is also preserved. Thus, $(-\x_A^{\top}, \x_B^{\top})^{\top}$  remains optimal. Any non-trivial convex combination of two distinct rank-1 optimizers is a feasible positive semidefinite matrix of higher rank that also attains the same optimal objective value, thereby demonstrating the existence of higher-rank optimal matrix.
\end{proof}
}

\subsection{Proof of \Cref{claim:2}}{\label{proof2}}
Without loss of generality, let 
\[
V_A=\{1,\dots,m_1\}, \qquad V_B=\{m_1+1,\dots,m_1+m_2\},
\]
and let $d_B := \max \deg_{G[V_B]} (v)$ for $v \in V_B$. We take the primal cut $\x \in \R^{|V_K|}$ such that $\x = \ind_{V_A} - \ind_{V_B}$ and the diagonal dual 
\[
\Y^* = \operatorname{diag}\!\left(\underbrace{\dfrac{m_2}{2},\dots,\dfrac{m_2}{2}}_{m_1},
\underbrace{\dfrac{m_1}{2},\dots,\dfrac{m_1}{2}}_{m_2}\right).
\]
Clearly,
\[
z_{\text{MAXCUTDUAL}} = \mathrm{tr}(\Y^*) = m_1 m_2 = z_{\text{MAXCUT}}.
\]

The dual slack takes block form
\[
\bS^* = \Y^* - \tfrac14\,\L_{G'} 
= \begin{bmatrix} \A & \frac{1}{4}\1_{m_1\times m_2} \\ \frac{1}{4}\1_{m_2\times m_1} & \B \end{bmatrix},
\]
with
\[
\A = \tfrac14(m_2 \I_{m_1} - \L_{G_A}), \qquad 
\B = \tfrac14(m_1 \I_{m_2} - \L_{G_B}).
\]
Since $\lambda_{\max}(\L_{G_B}) \le 2d_B$, the assumption $2d_B \le m_1$ ensures
\[
\lambda_{\min}(\B) 
\;\ge\; \dfrac{m_1 - \lambda_{\max}(\L_{G_B})}{4} \;\ge\; 0,
\]
which implies that $\B\succeq \bz$. To conclude that $\bS^*\succeq \bz$, we check its pseudo-Schur complement with respect to $\B$:
\begin{equation} \label{schur1}
\A - \1_{m_1 \times m_2} \B^{+} \1_{m_1 \times m_2}
= \tfrac14(m_2 \I_{m_1} - \L_{G_A}) 
- \tfrac{1}{16}\,\1_{m_1}\left(\1_{m_2}^\top \B^{+}\1_{m_2}\right)\1_{m_1}^\top.
\end{equation}

Since $\L_{G_B}\1_{m_2}=0$, we have $\B\1_{m_2}=\tfrac{m_1}{4}\1_{m_2}$ and hence 
\begin{equation}\label{schur2}
 \1_{m_2}^\top \B^{+}\1_{m_2}=\dfrac{4m_2}{m_1}   
\end{equation}
Substituting  \Cref{schur2} in \Cref{schur1} gives
\[
\A - \1_{m_1 \times m_2} \B^{+} \1_{m_2 \times m_1}
= \tfrac14(m_2 \I_{m_1} - \L_{G_A}) 
- \tfrac{m_2}{4m_1}\,\1_{m_1}\1_{m_1}^\top.
\]

For any vector $\bv \in \1_{m_1}^\perp$, we have $\1_{m_1}^\top \bv = 0$, and consequently 
\[
\1_{m_1}\1_{m_1}^\top \bv = \bz.
\]
Applying the pseudo-Schur complement to such vectors yields
\[
\bv^\top \!\left(\tfrac{1}{4}\bigl(m_2 \I_{m_1} - \L_{G_A}\bigr)\right) \bv,
\]
Since $\lambda_{\max}(\L_{G_A}) \le m_1 < m_2$, $\!\left(\tfrac{1}{4}\bigl(m_2 \I_{m_1} - \L_{G_A}\bigr)\right)$ matrix is positive definite on $\1_{m_1}^\perp$.

Furthermore, note that $\L_{G_A}\1_{m_1} = \bz$. Hence,
\[
\left(\tfrac{1}{4}\bigl(m_2 \I_{m_1} - \L_{G_A}\bigr) 
- \tfrac{m_2}{4m_1}\,\1_{m_1}\1_{m_1}^\top\right)\1_{m_1} = \bz,
\]
which shows that $\1_{m_1}$ lies in the nullspace of the pseudo-Schur complement. Hence the pseudo-Schur complement is positive semidefinite, and with $\range(\1_{m_2\times m_1}) \subseteq \range(\B)$ and $\B\succeq \bz$, we conclude that $\bS^*\succeq \bz$. Dual feasibility and strong duality then imply that $\X^*=\x\x^\top$ is an optimal primal solution, proving exactness.

For the uniqueness argument, first note that the leading block  
\[
\A=\tfrac14(m_2 \I_{m_1}-\L_{G_A})
\]
is positive definite whenever \(m_1 < m_2\), since
\[
\lambda_{\min}(\A) \;=\; \frac{m_2 - \lambda_{\max}(\L_{G_A})}{4} \;>\; 0,
\]
using \(\lambda_{\max}(\L_{G_A}) \le m_1 < m_2\). Hence \( \dim \mathcal{N} (\A)=0\).

Following the same reasoning as in \Cref{proof1}, i.e. applying \Cref{sylvester}, and rank inequality for the pseudo-Schur complement 
\begin{equation}\label{guttman}
    \rank(\X^*) \;\le\; k_B - k_0 + 1,
\end{equation}
where \(k_B = \dim \mathcal{N}(\B)\) and  
$k_0 = \dim\left( \mathcal{N}(\X^*) \cap \mathcal{N}(\bS^*) \right)$. Also, observe that for any \(\mathbf{w} \in \mathcal{N}(\B)\), the vector \((\bz^{\top}, \mathbf{w}^{\top})^{\top}\) lies in both \(\mathcal{N}(\bS^*)\) and \(\mathcal{N}(\X^*)\). Indeed, we have
\[
(\bz^{\top}, \mathbf{w}^{\top}) (\ind_{V_A} - \ind_{V_B})
= -\,\mathbf{w}^{\top}\1_{m_2} = 0,
\]
since \(\B\1_{m_2} = \tfrac{m_1}{4}\,\1_{m_2}\) implies \(\mathbf{w}^{\top} \1_{m_2} = 0\).
Thus \(k_B \le k_0\). Combining this with \eqref{guttman} and infeasibility of the zero matrix yields \(\rank(\X^*) = 1\). Finally, any nontrivial convex combination of distinct rank-1 optima would yield a higher-rank solution, and therefore the optimal primal matrix \(\X^*\) is unique.\myQED

\subsection{Proof of \Cref{thm:1}}{\label{sec4.5}}
We use the following lemmas to prove \Cref{thm:1}: 

\begin{lemma}\label{lem:1}
The Laplacian of $G_1 \bullet G_2$ is given by:
$$\L_{G_1 \bullet G_2} = \I_m \otimes\mathbf{D}_{G_2} + \mathbf{D}_{G_1} \otimes n \I_n + \I_m \otimes (\L_{G_2} - \mathbf{D}_{G_2}) + (\L_{G_1} - \mathbf{D}_{G_1}) \otimes \1_{n \times n},$$ where $\mathbf{D}_{G}$ is the degree matrix of graph $G$, $G_1$ is a graph on $m$ vertices and $G_2$ is a graph on $n$ vertices.
\end{lemma}
\begin{proof}
By construction of the lexicographic product, the adjacency matrix $\mathbf{A}$ of $G_1 \bullet G_2$ is $\mathbf{A}_{G_1 \bullet G_2} =
\I_m \otimes \mathbf{A}_{G_2} + \mathbf{A}_{G_1}\otimes \1_{n \times n}$ and the degree matrix of $G_1 \bullet G_2$ is $\mathbf{D}_{G_1 \bullet G_2} = \I_m \otimes \mathbf{D}_{G_2} + \mathbf{D}_{G_1} \otimes n \I_n$, where $m,n$ are cardinalities of vertex sets of $G_1$ and $G_2$, respectively. Proof follows from the identity $\L = \mathbf{D} - \mathbf{A}$.
(See \citep{barik2015laplacian} for details on spectra of lexicographic product).
\end{proof}

\begin{lemma}\label{lem:2}
Let $G(V,E)$ be an exact unweighted graph, and let $(\overline{\X}, \overline{\bS})$ be an optimal primal--dual pair of the $\mathrm{MAXCUT\ SDP}$ such that $\overline{\X}$ is rank-1. 
Define $\Y := \overline{\bS} + \tfrac{1}{4}\L_G$. 
Then for every vertex $i \in V$,
\[
Y_{ii} = \frac{\delta_i}{2},
\]
where $\delta_i = \bigl|\{\, j \in V : (i,j) \in E,\; x_i \neq x_j \,\}\bigr|$ 
is the number of edges incident on vertex $i$ crossing the cut induced by $\x$. 
Moreover, $\delta_i \ge \tfrac{\deg_G(i)}{2}$.
\end{lemma}

\begin{proof}
Since $G$ is exact, $\overline{\X} = \x\x^{\top}$ for some cut vector $\x \in \{\pm 1\}^n$ attaining the maximum cut. 
Optimality implies $\overline{\bS} \succeq \bz$ and $\x^{\top}\overline{\bS}\x = 0$, hence $\overline{\bS}\x = \mathbf{0}$. 

From $\overline{\bS} = \Y - \tfrac{1}{4}\L_G$, we obtain $\Y\x = \tfrac{1}{4}\L_G\x$, and since $\Y$ is diagonal,
\[
Y_{ii}x_i = \dfrac{1}{4}\!\left(\deg_G(i) x_i - \sum_{ij \in E} x_j\right)
\quad\Rightarrow\quad
Y_{ii} = \dfrac{1}{4}\!\left(\deg_G(i) - \sum_{ij \in E} x_i x_j\right).
\]

Let $\delta_i^c := \deg_G(i) - \delta_i$. Substituting $\deg_G(i) = \delta_i + \delta_i^c$ and $\displaystyle \sum_{ij \in E} x_i x_j = \delta_i^c - \delta_i$ yields
\[
Y_{ii}
= \dfrac{1}{4}\bigl((\delta_i + \delta_i^c) - (\delta_i^c - \delta_i)\bigr)
= \dfrac{1}{4}\bigl(\delta_i + \delta_i^c + \delta_i - \delta_i^c\bigr)
= \frac{\delta_i}{2}.
\]

Finally, since $\overline{\bS} = \Y - \tfrac{1}{4}\L_G \succeq \bz$, we have
\[
S_{ii} = \frac{\delta_i}{2} - \frac{\deg_G(i)}{4} \ge 0,
\]
which implies $\delta_i \ge \dfrac{\deg_G(i)}{2}$.
\end{proof}

We are now ready to prove \Cref{thm:1}.

\begin{proof}

Identify each vertex of \( G_1 \bullet G_2 \) as \( x_{ij} \), where \( i \in V_1 \) and \( j \in V_2 \).
Assume that \( G_1 \) is connected; otherwise, the argument extends directly to each connected component.

Let \((\overline{\X}, \overline{\bS})\) denote an optimal primal–dual pair for the exact graph \( G_1 \); and let \((V_A, V_B)\) be the corresponding partition attaining \(\mc(G_1)\).
Define the vector \(\hat{\x} \in \{\pm 1\}^{|V_1||V_2|}\) by
\[
\hat{x}_{ij} =
\begin{cases}
+1, & i \in V_A, \\[3pt]
-1, & i \in V_B,
\end{cases}
\qquad \text{for all } j \in V_2,
\]
and set \(\hat{\X} = \hat{\x}\hat{\x}^\top.\)
We claim that \(\hat{\X}\) is an optimal primal solution for \( G_1 \bullet G_2 \).

Since each copy of \( G_2 \) associated with vertex \( i \) inherits the same sign as its parent vertex in \( G_1 \), 
the edges within every \( G_2 \) do not contribute to the cut under \(\hat{\x}\), and the edges between different copies follow the cut induced by \((V_A, V_B)\).
Hence, the total cut value scales by \( n^2 \), where \( n = |V_2| \), giving
\[
z_{\mathrm{MAXCUT}}(G_1 \bullet G_2) = n^2 \, \mc(G_1).
\]
    
    Let $\overline{\Y} = \overline{\bS} + \frac{1}{4} \L_{G_1}$. Take $\hat{\Y} = \overline{\Y} \otimes n \I_n$. We claim that $\hat{\bS} = \hat{\Y} - \frac{1}{4}\L_{G_1 \bullet G_2}$ is the dual optimal solution. Note that $\tr(\hat{\Y}) = n^2 \tr(\overline{\Y}) = n^2 \mc(G_1)$. We reformulate $\hat{\bS}$ using \Cref{lem:1}.
\begin{align*}
    \hat{\bS} &= \overline{\Y} \otimes (\L_{K_n} + \1_{n \times n})  
    - \frac{1}{4} \left(\mathbf{D}_{G_1} \otimes (\L_{K_n} + \1_{n \times n})  
    + \I_m \otimes (\L_{G_2} - \mathbf{D}_{G_2})\right. \notag \\
    &\left.\quad + (\L_{G_1} - \mathbf{D}_{G_1}) \otimes \1_{n \times n}  
    + \I_m \otimes \mathbf{D}_{G_2}  
    \right) \\
    &= \left(\overline{\Y} - \frac{1}{4}\mathbf{D}_{G_1} \right) \otimes \L_{K_n}  
    + \left(\overline{\Y} - \frac{1}{4} \L_{G_1}\right) \otimes \1_{n \times n}  
     - \frac{1}{4} \I_m \otimes (\L_{K_n} - \L_{{G_2}^c})\\
    &= \underbrace{\left(\overline{\Y} - \frac{1}{4}\mathbf{D}_{G_1}-\frac{1}{4}\I_m\right) \otimes \L_{K_n}}_{\Y_1}  
    + \underbrace{\left(\overline{\Y} - \frac{1}{4} \L_{G_1}\right) \otimes \1_{n \times n}}_{\Y_2}  
    + \underbrace{\frac{1}{4} \I_m \otimes \L_{{G_2}^c}}_{\Y_3}.
\end{align*}

where the penultimate equality follows since, for any graph \(G\) on \(n\) vertices, 
\(\L(G) + \L(G^c) = \L(K_n)\). Since $\overline{\Y}$ is optimal for $G_1$, $\lambda_{\min}(\Y_2) = 0$ and it is clear that $\lambda_{\min}(\Y_3) = 0$. Referring to \Cref{lem:2}, note that for every vertex $i$ in an exact graph without isolated vertices $G$, $\Y_{ii} - \frac{1}{4} \mathbf{D_G}_{ii} \geq \frac{1}{4}$. It follows that $\lambda_{\min}(\Y_1) = 0$ and therefore $\hat{\bS} \succeq \bz$.

Given any rank-$r$ optimal (exact) solution $\X$ for $G_{1}$, we lift it to the product by setting $\widehat{\X}=\X\otimes \1_{n \times n}$. Since $\mathrm{rank}(\1_{n \times n})=1$, the rank is preserved: $\mathrm{rank}(\widehat{\X})=\mathrm{rank}(\X)=r$. We have $z_{\mc} = n^2 \mc(G_1)$ for $\hat{\X}$. The dual solution remains the same as $\hat{\mathbf{S}}$ earlier. The result follows.
\end{proof}

\subsection{Proof of \Cref{thm:3}}\label{counterrank}

The next sequence of additional notation and results, inspired by the framework of Delorme--Poljak~(\cite[Theorem~3.2]{delorme1993combinatorial}), develops the machinery to explicitly construct the required matrix $\X_G$.

For each \(i=1,\dots,n\) define the vectors \(\bu^i\in\mathbf R^n\) by
\[
u^i_j :=
\begin{cases}
1, & j\ne i\\[6pt]
-\,\dfrac{M-m_i}{m_i}, & j=i
\end{cases}
\quad,
\]
where $M = \sum_{i=1}^n m_i$. Write \( \tilde \bu^i := (\bu^i)^{\circ2}\) for the entry-wise square of \(u^i\).

\begin{lemma}\label{lem2}
For $\bm$ non-dominating, there exists \(\mathbf{d'}=(d'_1,\dots,d'_n)^\top>0\) such that
\[
\sum_{i=1}^n d'_i\, \tilde \bu^i \;=\; \1_n.
\]
\end{lemma}

\begin{proof}

If $\mathbf{m}$ is non-dominating, then $(u^i)^{\circ2}_j = 1$ for $j\ne i$ and $(u^i)^{\circ2}_i = \left(\tfrac{M-m_i}{m_i}\right)^2 = 1+\epsilon_i$ with $\epsilon_i>0$. Writing $\tilde u^i_i = 1+\epsilon_i$, define
\[
t := \frac{\sum_{i} \epsilon_i^{-1}}{1 + \sum_{k=1}^n \epsilon_k^{-1}},
\qquad
d'_i := \frac{1-t}{\epsilon_i}.
\]
Then $d'_i>0$, and for each coordinate $j \in [n]$,
\[
\sum_{i=1}^n d'_i \tilde u^i_j
= \sum_{i\ne j} d'_i + d'_j(1+\epsilon_j)
= \left(\sum_{i=1}^n d'_i\right) + d'_j\epsilon_j = 1,
\]
by the construction of $t$ and $d'_i$. Hence $\mathbf{d'}> \bz$ satisfies the claim.
\end{proof}

\begin{lemma}\label{lem3}
The vectors \(\bu^1,\dots,\bu^n\) span \(\bm^{\perp}\). Equivalently the \(n\times n\) matrix \(\mathbf{U}=[\bu^1\ \cdots\ \bu^n]\) has \(\operatorname{rank}(\mathbf{U})=n-1\).
\end{lemma}

\begin{proof}

For each $i \in [n]$ we have $\bu^i \in \bm^{\perp}$, since
\[
\bm^{\top} \bu^i = \sum_{j\ne i} m_j + m_i\left(-\frac{M-m_i}{m_i}\right) = (M-m_i) - (M-m_i)=0.
\]
Suppose \(\mathbf{z}\in\mathbb R^n\) satisfies \(\mathbf{z}^\top \bu^i=0\) for all \(i\). Let \(Z:=\sum_{j=1}^n z_j\). For fixed \(i\),
\[
0 = \mathbf{z}^\top \bu^i = \sum_{j\ne i} z_j + z_i\left(-\frac{M-m_i}{m_i}\right)
= Z - \frac{M}{m_i} z_i.
\]
Hence \(z_i = \dfrac{m_i}{M}Z\) for every \(i \in [n]\), so \(\mathbf z=(Z/M)\bm\). Thus the left nullspace of \(\mathbf{U}\) equals \(\operatorname{span}\{\bm\}\) (one-dimensional), and \(\operatorname{rank}(\mathbf{U})=n-1\). Therefore \(\operatorname{span}\{\bu^i\}=\bm^\perp\).
\end{proof}

\Cref{lem2} and \Cref{lem3} show that the vectors $\{\bu^i\}$ span $\bm^\perp$. These vectors are utilized to explicitly construct the rank-$(n-1)$ optimal solution required in the proof of \Cref{thm:3}, as follows.

\begin{proof}
Construct vectors $\bu^i\in \bm^\perp, i \in [n]$ as before, with $u^i_j=1$ for $j\neq i$ and $u^i_i=-(M-m_i)/m_i$. By \Cref{lem2}, there exist strictly positive weights $d'_i, i \in [n]$ such that $\sum_i d'_i (\bu^i)^{\circ2}=\1_n$. Define
\[
\X_G := \sum_{i=1}^n d'_i\, \bu^i (\bu^i)^\top.
\]
It immediately follows that, $\X_G \succeq \bz$  being the sum of positive semidefinite matrices and satisfies $\operatorname{diag}(\X_G)=\1_n$, and is therefore feasible for the $\mathrm{MAXCUT\ SDP}$. From \Cref{lem3}, the vectors $\{\bu^i\}$ span $\bm^\perp$, hence the range of $\X_G$ is $\bm^\perp$ and $\operatorname{rank}(\X_G)=n-1$.

Optimality is certified by the dual solution $\Y = \tfrac M4 \operatorname{Diag}(\bm)$. The Laplacian satisfies 
\begin{equation*}
\L_{G} = M\operatorname{Diag}(\bm) - \bm \bm{^\top},  
\end{equation*}
whose entries are $(\L_{G})_{ii} = m_i(M - m_i)$ and $(\L_{G})_{ij} = -m_i m_j$ for $i \neq j$, implying $\bS = \Y - \tfrac14 \L_{G} = \tfrac14 \bm \bm{^\top} \succeq \bz$. Complementary slackness holds because $\langle \X_G, \bS \rangle = 0$, as $\X_G \bm = \bz$, and the objective value is $\tr(\Y) = M^2 / 4$. Therefore, $\X_G$ is an optimal rank-$(n-1)$ solution with objective value $\tfrac14 M^2$.
\end{proof}

\subsubsection{Proof of \Cref{convex_hull}}
The claim follows immediately from the existence of an optimal solution of rank $n-1$ and the fact that any convex combination of $n-2$ rank-1 solutions has rank at most $n-2$.\hfill \myQED

\subsubsection{Verification of \Cref{counter1} and \Cref{counter2}}\label{verification}
For \Cref{counter1}, take the four vectors
\[
\bu^1 =(-\tfrac{11}{5},\,1,1,1)^{\top}, 
\bu^2 =(1,\,-\tfrac{13}{3},\,1,1)^{\top}, 
\bu^3 =(1,1,\,-3,\,1)^{\top},
\bu^4 =(1,1,1,\,-3)^{\top}.
\]
Applying the construction in \Cref{lem2} gives

\[
d_1 = \tfrac{125}{752}, \qquad
d_2 = \tfrac{27}{752}, \qquad
d_3 = d_4 = \tfrac{15}{188}.
\]
Writing \(\X_{G_1} = \sum_{i=1}^4 d_i\, \bu^i (\bu^i)^\top\), we obtain
\[
\X_{G_1} =
\begin{bmatrix}
1 & -\tfrac{17}{47} & -\tfrac{23}{47} & -\tfrac{23}{47} \\[2mm]
-\tfrac{17}{47} & 1 & -\tfrac{7}{47} & -\tfrac{7}{47} \\[2mm]
-\tfrac{23}{47} & -\tfrac{7}{47} & 1 & -\tfrac{13}{47} \\[2mm]
-\tfrac{23}{47} & -\tfrac{7}{47} & -\tfrac{13}{47} & 1
\end{bmatrix}.
\]
If \(\X^* = \x \x^\top\) is the rank-1 optimal solution corresponding to \(\x = (1,1,-1,-1)^{\top}\), then
\(\hat \bZ^*\) lies in the convex combination of \(\X^*\) and \(\X_{G_1}\), given by
\[
\hat \bZ^* = \tfrac{1}{48} \X^* + \tfrac{47}{48} \X_{G_1}.
\]

\bigskip

For \Cref{counter2}, take the four vectors
\[
\bu^1 = (-5,\,1,1,1)^{\top}, \;
\bu^2 = (1,\,-5,\,1,1)^{\top}, \;
\bu^3 = (1,1,\,-2,1)^{\top}, \;
\bu^4 = (1,1,1,\,-2)^{\top}.
\]
Applying the construction in \Cref{lem2} (with \(\tilde u^i_j = (u^i_j)^2\)) gives

\[
d_1 = d_2 =  \frac{1}{42}, 
\qquad
d_3 = d_4 =  \frac{4}{21}.
\]
Writing \(\X_{G_2} = \sum_{i=1}^4 d_i\, \bu^i (\bu^i)^\top\), we obtain the rational matrix
\[
\X_{G_2} =
\begin{bmatrix}
1 & \tfrac{1}{7} & -\tfrac{2}{7} & -\tfrac{2}{7} \\[2mm]
\tfrac{1}{7} & 1 & -\tfrac{2}{7} & -\tfrac{2}{7} \\[2mm]
-\tfrac{2}{7} & -\tfrac{2}{7} & 1 & -\tfrac{5}{7} \\[2mm]
-\tfrac{2}{7} & -\tfrac{2}{7} & -\tfrac{5}{7} & 1
\end{bmatrix}.
\]
If \(\X^* = \x_1 \x_1^\top\) and \(\Y^* = \x_2 \x_2^\top\) are the rank-1 solutions corresponding to 
\(\x_1 = (1,-1,-1,1)^{\top}\) and \(\x_2 = (1,-1,1,-1)^{\top}\), then the matrix \(\hat \bZ\) admits the convex decomposition
\[
\hat \bZ = \tfrac{3}{20} \X^* + \tfrac{3}{20} \Y^* + \tfrac{7}{10} \X_{G_2},
\]
as claimed. Observe that $\X_{G_2}$ in this case, itself lies outside the convex hull of rank-1 optima.

\subsection{Proof of \Cref{lem6}}

This result follows by combining earlier results of Delorme and Poljak
\cite[Lemma~2.1]{delorme1993performance} and \cite[Theorem~2.7]{delorme1993combinatorial}.
Since the proof of \cite[Lemma~2.1]{delorme1993performance} was omitted in their paper, we provide the missing argument here.

\begin{proof}
Let \( G \) be exact, i.e., \( \phi(G) = \mc(G) \). By the graph vertex-split invariance and upper property of \( \phi \), we have $\phi(G) =\phi(\widetilde{G}) \geq \mc(\widetilde{G}) $. Moreover, every cut \( (S, S^c) \) of \( G \) naturally induces a cut
\( (\widetilde{S}, \widetilde{S}^c) \) of \( \widetilde{G} \), obtained by assigning all
vertices arising from the split of each vertex \( v \in S \) (respectively,
\( v \in S^c \)) to \( \widetilde{S} \) (respectively, \( \widetilde{S}^c \)). This shows that $\mc(\widetilde{G}) \ge \mc(G)$, and therefore 
\( \mc(G) = \phi(\widetilde{G}) \geq \mc(\widetilde{G}) \geq \mc(G) \).

For the converse, suppose that two vertices \( v', v'' \) of the vertex-split graph
\( \widetilde{G} \), originating from the same vertex \( v \in V(G) \),
lie on opposite sides of a maximum cut of \( \widetilde{G} \).
We claim that one can move either \( v' \) or \( v'' \) to the other side
without decreasing the cut value.

Let \( v' \) be incident to \( n \) vertices on its own side of the cut
and to \( m \) vertices on the opposite side.
Then \( v'' \) is incident to \( m \) vertices on its own side and to
\( n \) vertices on the opposite side, since \( v' \) and \( v'' \) have
identical neighborhoods.
If \( n \ge m \), moving \( v' \) to the side containing \( v'' \) does not
decrease the cut value; and vice-versa.

Repeating this process, we may assume that all vertices arising from the
vertex-split of each \( v \in V(G) \) lie on the same side of the cut, which
induces a cut of \( G \).
Hence, $\mc(\widetilde{G}) \le \mc(G)$, and the proof follows.
\end{proof}

\subsection{Proof of \Cref{cor:complete}}

In this section, we establish \Cref{splitlemma} to formalize how the primal solution rank and objective value behave under the vertex-splitting operation. We then leverage this lemma to prove \Cref{cor:complete}. Finally, we deploy this complete $k$-partite framework to resolve the necessity side of the argument in the uniqueness proof of \Cref{prop1}. 

\subsubsection{Proof of \Cref{splitlemma}}

Let $\X^*$ be an optimal rank-$r$ solution of the original graph $G(V,E)$ and $|V| = n$. 
Since $\X^* \succeq 0$, it admits a factorization $\X^* = \mathbf{V}\mathbf{V}^\top$, where  
$\mathbf{V}^{\top} = [\,\bv_1, \ldots, \bv_n\,]\in \mathbb{R}^{r \times n}$.

We construct a feasible matrix $\widetilde\X$ for the vertex-split graph $\widetilde G$ as follows. 
For each vertex $u \in V$ having $m_u$ copies/clones $u^{(1)}, \dots, u^{(m_u)}$ in $\widetilde G$, 
we assign $\widetilde\bv_{u^{(i)}} := \bv_u$ for all $i \in [m_u]$ and define 
$\widetilde X_{ab} := \widetilde\bv_a^{\top}\widetilde\bv_b$ for all $a,b \in V(\widetilde G)$. 
Clearly, $\widetilde\X \succeq \bz$, and since the span of the vectors 
$\{\widetilde\bv_a : a \in V(\widetilde G)\}$ coincides with that of $\{\bv_u : u \in V\}$, 
the rank is preserved, i.e., $\operatorname{rank}(\widetilde\X) = r$.

Next, we verify that the objective value is also preserved. 
For the diagonal terms in $\langle \L_G, \X^* \rangle$, 
let the total incident weight at vertex $u$ be $\sum_{v} w_{uv}$. 
In $\widetilde G$, each clone $u^{(i)}$ has total incident weight 
$\left(\sum_{v} w_{uv}\right) / m_u$. 
Hence, the total diagonal contribution of all clones of $u$ equals
\[
\sum_{i=1}^{m_u} 
\left(\dfrac{\sum_{v} w_{uv}}{m_u}\right)\, \widetilde X_{u^{(i)}u^{(i)}}
\;=\;
\left(\sum_{v} w_{uv}\right) X^*_{uu},
\]
matching the corresponding diagonal term in $\langle \L_G, \X^* \rangle$. 

For the off-diagonal terms, the weight $w_{uv}$ of each edge $(u,v) \in E$ 
is evenly distributed among the $m_u m_v$ clone--clone pairs 
$(u^{(i)}, v^{(j)})$ in $\widetilde G$. 
Each pair contributes $(w_{uv}/m_u m_v)\widetilde X_{u^{(i)}v^{(j)}} 
= (w_{uv}/m_u m_v) X^*_{uv}$, 
and summing over all $i \in [m_u]$ and $j \in [m_v]$ gives $w_{uv} X^*_{uv}$. 
Hence, the total contribution of all clone pairs exactly reproduces 
the off-diagonal part of the original inner product.

Combining both parts, we obtain 
$\langle \L_{\widetilde G}, \widetilde\X \rangle 
= \langle \L_G, \X^* \rangle$. 
Finally, since the SDP optimum is preserved under vertex-splitting, i.e. 
$\phi(\widetilde G) = \phi(G)$, and $\widetilde\X$ attains this value, $\widetilde\X$ is a rank-$r$ optimal solution for $\widetilde G$. \hfill \myQED

\subsubsection{Proof of the Main Statement (\Cref{cor:complete})}

We begin by establishing that the exactness of the complete $n$-partite graph framework can be verified through two distinct, equivalent routes:
\begin{enumerate}
    \item The balanced non-dominating case and the dominating case can both be viewed as special cases of \Cref{roc}.
    \item Alternatively, exactness follows from \Cref{lem6} and the observation that a complete $n$-partite graph can be viewed as the vertex-split graph of the exact weighted complete graphs identified by Laurent-Poljak \cite{laurent1995positive} (see \Cref{knowngraphs}).
\end{enumerate}

The respective optimal Max-Cut objective values can be directly verified from \Cref{proof1} and \Cref{proof2}.
For the uniqueness argument, suppose that uniqueness of the rank-$1$ optimal solution holds for the complete $n$-partite graph in the balanced non-dominating case. Then, by \Cref{thm:3} and \Cref{splitlemma}, the corresponding $\mathrm{MAXCUT\ SDP}$ admits an optimal solution of rank $n-1$, which is a contradiction. Hence, uniqueness can hold only in the dominating case. The sufficiency of the dominating condition follows immediately from \Cref{roc}. \hfill \myQED

\subsubsection{Proof of \Cref{prop1} continued}\label{onlyif}

Suppose neither $G_A^c$ nor $G_B^c$ is connected. Let the connected components of $G_A^c$ partition $V_A$ into $k \ge 2$ sets with sizes $a_1, a_2, \dots, a_k$, and let the connected components of $G_B^c$ partition $V_B$ into $l \ge 2$ sets with sizes $b_1, b_2, \dots, b_l$. 

The union graph $G'$ necessarily contains the complete $(k+l)$-partite graph $S' = K(a_1, \dots, a_k, b_1, \dots, b_l)$ as a spanning subgraph. Let $S$ denote the index set corresponding to the partitions of $V_A$, and its complement $S^c$ denote the index set corresponding to the partitions of $V_B$. Since $\sum_{i \in S} a_i = |V_A| = n$ and $\sum_{j \in S^c} b_j = |V_B| = n$, it follows that:
\[
\sum_{i \in S} a_i = \sum_{j \in S^c} b_j = n.
\]

This precisely satisfies the condition for the balanced non-dominating case \Cref{cor:complete}. Thus, the spanning subgraph $S'$ admits a higher-rank optimal primal solution $X^*$ with an optimal objective value of $\left(\sum_{i \in S} a_i\right)^2 = n^2$. To see that $X^*$ remains optimal for the full graph $G'$, recall from \Cref{proof1} that dual feasibility is preserved with an optimal dual objective value of $n^2$ even when the remaining internal edges of $G_A$ and $G_B$ are included. Since a higher-rank $X^*$ is primal feasible for $G'$ and achieves an objective value of $n^2$ matching the dual bound, the same primal solution $X^*$ is optimal for $G'$. Hence, the optimal SDP solution is not unique. \hfill \myQED

\subsection{Proof of \Cref{thm:splitexact}} 
\label{proof:prop5}
Let $\mathcal{S}'$ denote the uniform vertex-split of $\mathcal{S}$ with $p$ copies for each vertex. Since $\mathcal{S}$ has uniform edge weights $w$, the resulting graph $\mathcal{S'}$ is uniformly weighted with edge weights $w/p^{2}$. By Lemma~\ref{lem6}, the exactness of $\mathcal{S}$ implies the exactness of $\mathcal{S}'$, and moreover we obtain 
\begin{equation*}
 \phi(\mathcal{S}') \;=\; \phi(\mathcal S) \;=\; \mc(\mathcal{S}) \;=\; w \cdot \mc(\mathcal{S}_{\mathrm{uw}})   
\end{equation*}
where $\mc(\mathcal{S}_{\mathrm{uw}})$ denotes the maximum cut value of the unweighted copy of $\mathcal{S}$.

Let $\mathcal{S}_{\tfrac{w}{p^2}}$ be the graph such that
\[
V(\mathcal{S}_{\tfrac{w}{p^2}}) = V(\mathcal{S}), 
\qquad 
E(\mathcal{S}_{\tfrac{w}{p^2}}) = E(\mathcal{S}), 
\qquad
w_{ij}(\mathcal{S}_{\tfrac{w}{p^2}}) = \tfrac{w}{p^2}, \ \forall (i,j) \in E(\mathcal{S}).
\]
Similarly, let $K_{p,\tfrac{w}{p^2}}$ be the complete graph on $p$ vertices with all edges uniformly weighted by $\tfrac{w}{p^2}$. Now, consider the lexicographic product $\mathcal{S}_{\tfrac{w}{p^2}} \bullet K_{p,\tfrac{w}{p^2}}$. Scaling both $G_1$ and $G_2$ by any constant weight $\lambda$ scales the Laplacian of their lexicographic product by the same factor $\lambda$, which can be cross-verified using Lemma~\ref{lem:1}. By Theorem~\ref{thm:1}, this product is exact and satisfies 
\begin{equation*}
 \phi\left(\mathcal{S}_{\tfrac{w}{p^2}} \bullet K_{p,\tfrac{w}{p^2}}\right) \;=\; \frac{w}{p^2} \cdot p^{2} \mc(\mathcal{S}_{\mathrm{uw}}) \;=\; w \cdot \mc(\mathcal{S}_{\mathrm{uw}})   
\end{equation*}

Observe that $\mathcal{S}'$ is a spanning subgraph of $\mathcal{S}_{\tfrac{w}{p^2}} \bullet K_{p,\tfrac{w}{p^2}}$, while $\phi\left(\mathcal{S}_{\tfrac{w}{p^2}} \bullet K_{p,\tfrac{w}{p^2}}\right) = \phi(\mathcal{S}')$. Since $\widetilde{G}$ lies between these two extremes edge-wise, we can invoke the monotonicity property of $\phi$. It follows that $\widetilde{G}$ is also exact with 

\begin{equation*}
 \phi(\widetilde{G}) \;=\; \mc(\mathcal{S}) \;=\; w \cdot \mc(\mathcal{S}_{\mathrm{uw}}).  
\end{equation*}

From \Cref{lem6}, this implies that $G$ is exact. \hfill \myQED
\section{Concluding Remarks}\label{sec6} 

We investigated structural and algorithmic aspects of exactness of the Goemans--Williamson (GW) relaxation for the Max-Cut problem. We identified two unweighted graph families for which the GW-relaxation is exact:
\begin{enumerate}[label=\alph*)]
    \item graphs containing a balanced complete bipartite spanning subgraph, and 
    \item graphs admitting a partition with an adjacency condition and a degree constraint. 
\end{enumerate}
For the graph class (a), we established sufficient conditions under which the optimal SDP solution is unique, while for the latter class we showed that the optimal SDP solution is always unique. Furthermore, we presented an efficient method to recognize these classes in \Cref{sec4} and provided explicit constructions of their maximum cut in \Cref{sec5}. This raises broader queries concerning which additional structural conditions enable efficient certification of exactness and direct recovery of optimal cuts, and how prevalent such graph classes are within natural or application-driven graph families. Particularly, since most of our results concern unweighted or uniformly weighted graphs, it would be interesting to investigate further the conditions
under which semidefinite relaxations yield exact solutions for non-uniformly weighted instances. 

We further examined the behavior of exactness and optimal solution rank under a few graph operations: which includes graph join, vertex-splitting and lexicographic product. Using these insights, we introduced a class of graphs—termed \emph{split-decomposable graphs}—to demonstrate how large graphs may remain exact due to the presence of a small structural core that governs exactness. This framework also shows how the tools developed in this manuscript can be leveraged to identify exactness and compute the maximum cut in larger and more complex graphs. At the same time, the computational complexity of recognizing exact unweighted graphs in full generality remains open, suggesting the development of explicit decomposition and recognition algorithms as a direction for future research.

Finally, we examined the structure of optimal solutions in exact relaxations and uncovered intrinsic higher-rank phenomena. In addressing open problems posed by Mirka and Williamson~\citep{mirka2024max}, we showed that exactness of the GW-relaxation together with uniqueness of the maximum cut does not imply uniqueness of the rank-1 optimal SDP solution. Furthermore, we constructed families of graphs for which higher-rank optimal SDP solutions exist that cannot be expressed as convex combinations of rank-1 optima. Understanding what additional structural principles govern the existence of such higher-rank optima—particularly those lying outside the convex hull of cut solutions—offers insight into the extreme points of the elliptope and the convex hull of optimal solutions. Advancing these lines of inquiry could enhance the applicability of the theoretical framework and provide deeper insights into the structural properties that
govern exactness of semidefinite relaxation for the Max-Cut problem.

\bibliographystyle{cas-model2-names}
\bibliography{references.bib}

\end{document}